\documentclass[11pt,a4paper,final]{article}

\usepackage[hmargin={25mm,25mm},vmargin={30mm,35mm}]{geometry}
\usepackage{amsfonts}
\usepackage{amssymb}
\usepackage{amsthm}
\usepackage{bm}
\usepackage{mathtools}
\usepackage{latexsym}
\usepackage{graphicx}
\usepackage{booktabs}
\usepackage{pgfplots,pgfplotstable}
\pgfplotsset{compat=1.16}
\usepackage{xcolor}
\usepackage{setspace}
\usepackage{url}
\usepackage{float}
\usepackage{subcaption}
\usepackage{enumitem}

\usepackage[colorlinks,allcolors={blue},linkcolor={black}]{hyperref}

\numberwithin{equation}{section}
\usepackage[color,notref,notcite]{showkeys}

\usepackage[normalem]{ulem}
\newcounter{corr}
\definecolor{violet}{rgb}{0.580,0.,0.827}
\newcommand{\corr}[4][]{\typeout{Warning : a correction remains in page \thepage}
 \stepcounter{corr}
	      {\color{blue}\ifmmode\text{\,\sout{\ensuremath{#2}}\,}\else\sout{#2}\fi}
        {\color{green!50!black}#3}
        {\color{violet}#4}
}

\usepackage{authblk}
\newcommand{\email}[1]{\href{mailto:#1}{#1}}

\newtheorem{theorem}{Theorem}

\newtheorem{lemma}[theorem]{Lemma}

\newtheorem{assumption}{Assumption}

\newtheorem{remark}{Remark}

\newcommand{\bmrm}[1]{{\bm{{\rm #1}}}}

\newcommand{\mat}[1]{\bmrm{#1}}

\newcommand{\bmn}{{\bm{n}}}

\newcommand{\bmv}{{\bm{v}}}

\newcommand{\bmx}{{\bm{x}}}

\newcommand{\calE}{\mathcal{E}}
\newcommand{\calF}{\mathcal{F}}

\newcommand{\calM}{\mathcal{M}}

\newcommand{\calT}{\mathcal{T}}

\newcommand{\bbN}{\mathbb{N}}

\newcommand{\bbP}{\mathbb{P}}

\newcommand{\bbR}{\mathbb{R}}

\newcommand{\bbZ}{\mathbb{Z}}

\newcommand{\rma}{{\rm{a}}}

\newcommand{\rmp}{{\rm{p}}}

\newcommand{\rms}{{\rm{s}}}

\newcommand{\eqs}{=}

\newcommand{\eq}{ ={}& }
\newcommand{\lea}{ \le{}& }

\newcommand{\les}{ \lesssim{}& }
\newcommand{\plus}{ &{}+ }
\newcommand{\minus}{ &{}- }

\newcommand{\nn}{\nonumber}
\newcommand{\nl}{\nn\\}

\newcommand{\defeq}{\vcentcolon=}

\newcommand{\deq}{ \defeq & \, }

\newcommand{\ul}[1]{\underline{#1}}
\newcommand{\ol}[1]{\overline{#1}}

\newcommand{\bdry}{\partial}

\newcommand{\subT}{{\text{$T$}}}
\newcommand{\subF}{{\text{$F$}}}
\newcommand{\subFT}{{\text{$\partial T$}}}
\newcommand{\subh}{{\text{$h$}}}
\newcommand{\xT}{{\bmrm{x},\subT}}
\newcommand{\xFT}{{\bmrm{x},\subFT}}
\newcommand{\xF}{{\bmrm{x},\subF}}
\newcommand{\xh}{{\bmrm{x},\subh}}

\newcommand{\subD}{{\text{$\Delta$}}}
\newcommand{\subN}{{\text{$\nabla$}}}

\newcommand{\Mh}[1][h]{\calM_{#1}}
\newcommand{\Th}[1][h]{\calT_{#1}}
\newcommand{\Fh}[1][h]{\calF_{#1}}

\newcommand{\Eh}[1][h]{\calE_{#1}}

\newcommand{\T}{{T}}
\newcommand{\F}{{F}}

\newcommand{\hT}{h_\T}

\newcommand{\bdryT}{{\bdry \T}}

\newcommand{\nor}{\bmn}
\newcommand{\norT}{\nor_{\bdryT}}
\newcommand{\norF}{\nor_{\F}}
\newcommand{\norTF}{\nor_{\T\F}}

\def\R{\bbR}
\def\Z{\bbZ}
\def\N{\bbN}

\newcommand{\POLY}[1]{\bbP^{#1}}
\newcommand{\POLYX}[1]{\bbP_{\bmrm{x}}^{#1}}
\newcommand{\POLYD}[1]{\bbP_{\subD}^{#1}}
\newcommand{\POLYN}[1]{\bbP_{\subN}^{#1}}
\newcommand{\HS}[1]{H^{#1}}
\newcommand{\HONE}{\HS{1}}
\newcommand{\HONEzr}{\HONE_0}

\newcommand{\LP}[1]{L^{#1}}
\newcommand{\LTWO}{\LP{2}}

\newcommand{\norm}[2][]{\|#2\|_{#1}}
\newcommand{\seminorm}[2][]{|#2|_{#1}}
\newcommand{\brac}[2][]{(#2)_{#1}}

\newcommand{\Brac}[2][]{\Big(#2\Big)_{#1}}

\newcommand{\energynorm}[1]{\|#1\|_{\a,\xh}}

\newcommand{\piT}[1]{\pi_\subT^{#1}}
\newcommand{\piTzr}[1]{\piT{0, #1}}
\newcommand{\piTe}[1]{\piT{1, #1}}

\newcommand{\piXTzr}[1]{\pi_{\subD,\subT}^{0,#1}}
\newcommand{\piXTe}[1]{\pi_{\xT}^{1,#1}}

\newcommand{\piXFzr}[1]{\pi_{\subN,\subF}^{0,#1}}
\newcommand{\piXFTzr}[1]{\pi_{\subN,\subFT}^{0,#1}}

\newcommand{\pXh}[1]{\rmp_{\xh}^{#1}}
\newcommand{\pXT}[1]{\rmp_{\xT}^{#1}}

\newcommand{\U}[2]{\ul{U}_{#1}^{#2}}

\newcommand{\UXTk}{\U{\xT}{k}}

\newcommand{\UXhk}{\U{\xh}{k}}
\newcommand{\UXhkzr}{\U{\xh,0}{k}}

\newcommand{\I}[2]{\ul{I}_{#1}^{#2}}

\newcommand{\Ih}[1]{\I{h}{#1}}
\newcommand{\Ihk}{\Ih{k}}

\newcommand{\IXTk}{\I{\xT}{k}}
\newcommand{\IXhk}{\I{\xh}{k}}

\def\a{\rma}

\newcommand{\aXT}{\a_{\xT}}
\newcommand{\sXT}{\rms_{\xT}}
\newcommand{\aXh}{\a_{\xh}}
\newcommand{\sXh}{\rms_{\xh}}

\newcommand{\uluh}{\ul{u}_h}

\newcommand{\vXT}{v_\xT}
\newcommand{\vXh}{v_\xh}
\newcommand{\vXF}{v_\xF}
\newcommand{\vXFT}{v_{\xFT}}

\newcommand{\uXT}{u_\xT}

\newcommand{\uXF}{u_\xF}

\newcommand{\ulvXT}{\ul{v}_\xT}
\newcommand{\uluXT}{\ul{u}_\xT}
\newcommand{\ulvXh}{\ul{v}_\xh}
\newcommand{\uluXh}{\ul{u}_\xh}

\newcommand{\deltaXT}[1]{\delta_\xT^{#1}}

\newcommand{\deltaXFT}[1]{\delta_{\xFT}^{#1}}

\newcommand{\locerrorRHS}[1]{\Big[\hT^{k+1}\seminorm[\HS{k+2}(T)]{#1}\Big]^2}

\newcommand{\logLogSlopeTriangle}[5]
{
	\pgfplotsextra
	{
		\pgfkeysgetvalue{/pgfplots/xmin}{\xmin}
		\pgfkeysgetvalue{/pgfplots/xmax}{\xmax}
		\pgfkeysgetvalue{/pgfplots/ymin}{\ymin}
		\pgfkeysgetvalue{/pgfplots/ymax}{\ymax}
		
		\pgfmathsetmacro{\xArel}{#1}
		\pgfmathsetmacro{\yArel}{#3}
		\pgfmathsetmacro{\xBrel}{#1-#2}
		\pgfmathsetmacro{\yBrel}{\yArel}
		\pgfmathsetmacro{\xCrel}{\xArel}
		
		\pgfmathsetmacro{\lnxB}{\xmin*(1-(#1-#2))+\xmax*(#1-#2)} 
		\pgfmathsetmacro{\lnxA}{\xmin*(1-#1)+\xmax*#1} 
		\pgfmathsetmacro{\lnyA}{\ymin*(1-#3)+\ymax*#3} 
		\pgfmathsetmacro{\lnyC}{\lnyA+#4*(\lnxA-\lnxB)}
		\pgfmathsetmacro{\yCrel}{\lnyC-\ymin)/(\ymax-\ymin)}
		
		\coordinate (A) at (rel axis cs:\xArel,\yArel);
		\coordinate (B) at (rel axis cs:\xBrel,\yBrel);
		\coordinate (C) at (rel axis cs:\xCrel,\yCrel);
		
		\draw[#5]   (A)-- node[pos=0.5,anchor=north] {\scriptsize{1}}
		(B)-- 
		(C)-- node[pos=0.,anchor=west] {\scriptsize{#4}} 
		cycle;
	}
}

\numberwithin{equation}{section}

 \usepackage[color,notref,notcite]{showkeys}
\begin{document}	

\title{Design and analysis of the Extended Hybrid High-Order method for the Poisson problem}
\author{Liam Yemm}
\affil{School of Mathematics, Monash University, Melbourne, Australia, \email{liam.yemm@monash.edu}}
\maketitle

\begin{abstract}
  We propose an Extended Hybrid High-Order scheme for the Poisson problem with solution possessing weak singularities. Some general assumptions are stated on the nature of this singularity and the remaining part of the solution. The method is formulated by enriching the local polynomial spaces with appropriate singular functions. Via a detailed error analysis, the method is shown to converge optimally in both discrete and continuous energy norms. Some tests are conducted in two dimensions for singularities arising from irregular geometries in the domain. The numerical simulations illustrate the established error estimates, and show the method to be a significant improvement over a standard Hybrid High-Order method.
\end{abstract}

	
\section{Introduction}
This century has seen a growing interest in so called polytopal methods for the discretisation of elliptic problems. These methods extend the classical Finite Element method to generic polytopal grids. A short list of such methods includes Discontinuous Galerkin and Hybridizable Discontinuous Galerkin methods \cite{di-pietro.ern:2011:mathematical,cangiani.dong.ea:2017:discontinuous,cockburn.dong.ea:2009:hybridizable}, Multi-Point Flux Approximation finite volume methods \cite{aavatsmark.eigestad.ea:2008:compact}, Hybrid Mimetic Mixed methods \cite{droniou.eymard:2010:unified}, Virtual Element methods \cite{beirao-da-veiga.brezzi.ea:2013:basic,ahmad.alsaedi.ea:2013:equivalent,brezzi.falk.ea:2014:basic,cangiani.manzini.ea:2017:conforming}, weak Galerkin methods \cite{mu.wang.ea:2015:weak}, and polytopal Finite Elements \cite{sukumar.tabarraei:2004:conforming}. The Hybrid High-Order (HHO) method is a recent addition to these techniques, and also benefits from being of arbitrary order and dimension independent. Additionally, HHO methods allow for static condensation of the system matrix, and depend on polynomial reconstructions which account for the physics of the problem. For a comprehensive discussion of the method and its applications we refer the interested reader to the monograph \cite{di-pietro.droniou:2020:hybrid}. 

Hybrid High-Order methods have been used to model diffusion and diffusion-advection-reaction equations \cite{di-pietro.ern.ea:2014:arbitrary,di-pietro.droniou.ea:2014:discontinuous}, elasticity problems \cite{di-pietro.ern:2015:hybrid, botti.di-pietro:2017:hybrid}, Leray Lions and p-Laplace equations \cite{di-pietro.droniou:2017:hybrid,di-pietro.droniou:2020:hybrid}, and the Stokes and Navier-Stokes equations \cite{di-pietro.ern.ea:2016:discontinuous,botti.di-pietro.ea:2019.hybrid,di-pietro.droniou:2020:hybrid}. The method has also been shown to be robust on highly irregular grids consisting of elements with arbitrarily many small faces \cite{droniou.yemm:2022:robust} and on skewed meshes such that each element possesses a linear map to an an isotropic element \cite{droniou:2020:interplay}. All of these schemes, however, rely on error estimates which assume certain regularity of the exact solution. This is typical of high-order approximations which require high-order regularity to obtain optimal approximation rates. On non-smooth domains (such as regions with non-convex corners or those possessing cracks) it is expected that the exact solution to elliptic problems will contain weak singularities \cite{grisvard:1985:elliptic}. This lack of regularity is well documented in the finite element literature and is typically overcome through enriched approximations based on the partition of unity method \cite{melenk.babuvska:1996:partition,babuvska.melenk:1997:partition}. The extended Finite Element method \cite{belytschko.black:1999:elastic,moes.dolbow:1999:finite} is one such method, originally designed to handle discontinuities in crack growth models. In particular, by enriching the local spaces with basis functions that are discontinuous across the crack, the method allows for optimal approximation without the need for mesh refinement near the discontinuity. A similar approach was taken for polygonal finite elements \cite{zamani.eslami:2011:embedded}. More recently, an enriched Virtual Element method for the Poission problem was designed in \cite{benvenuti.chiozzi.ea:2019:extended}, however no estimates of the error were given. Following this work, the same authors have proposed in \cite{benvenuti.chiozzi.ea:2019:extended} an enriched VEM for a linear elasticity fracture problem. The Hybrid High-Order method is closely linked (c.f. \cite[Section 5.5]{di-pietro.droniou:2020:hybrid}) to the nonconforming Virtual Element method (NCVEM) \cite{de-dios.lipnikov.ea:2016:nonconforming}. The recent article \cite{artioli.mascotto:2021:enrichment} designs an enriched NCVEM for harmonic singularities arising from irregular domains in two dimensions. Moreover, the enriched NCVEM is capable of handling highly irregular harmonic singularities, including those arising from cracked domains. While the assumptions we make in Assumption \ref{assum:regularity} do not cover cracked domains, the method presented here is robust for all other boundary singularities and has the particular advantage of not requiring the singular functions to be harmonic. This can be particularly useful if the irregularity of the problem is due to singularities in the source term, and not due to non-smooth domains. We also note that the analysis carried out in this paper does not require that inverse inequalities hold for the enriched polynomial spaces.

In this paper we propose an Extended Hybrid High-Order (XHHO) method for the Poisson problem. The work is presented dimension independent, and capable of handling arbitrary singular functions satisfying Assumption \ref{assum:regularity}. Specifically, we assume the exact solution consists of a `weakly singular' part lying in a finite dimensional singular space. The local polynomial spaces on the mesh elements and faces are enriched with the appropriate singular space. By then adjusting the local projectors and potential reconstruction accordingly, an optimal XHHO scheme is developed that mimics the standard method \cite[Section 2]{di-pietro.droniou:2020:hybrid}. However, the analysis is far more involved as we can no longer rely on the Lebesgue/Sobolev embeddings and discrete trace inequalities that apply to polynomial spaces. In Section \ref{sec:analysis} we provide a thorough error analysis of the scheme under minimal regularity assumptions. The paper is concluded with an analysis of the choice of stabilisation term in Section \ref{sec:stab} and a discussion on the implementation of the scheme, its numerical limitations and some benchmark tests in Section \ref{sec:numerical}.

\subsection{Model Problem}
Let us consider the typical Dirichlet problem in a polytopal domain \(\Omega\subset\R^d\), \(d \ge 2\):
find \(u\in\HONEzr(\Omega)\) such that
\begin{equation}\label{eq:weak.form}	
	\a(u,v) = \ell(v),\qquad\forall\,v\in \HONEzr(\Omega) ,
\end{equation}
where \(\a(u,v) \defeq \brac[\Omega]{\nabla u, \nabla v}\), \(\ell(v)\defeq\brac[\Omega]{f,v}\) for some \(f\in \LTWO(\Omega)\). Here and in the following, \(\brac[X]{\cdot, \cdot} \) is the \(\LTWO\)-inner product of scalar- or vector-valued functions on a set \(X\) for its natural measure. To ease the analysis we consider only homogeneous Dirichlet boundary conditions in this paper. However, the scheme extends quite naturally to more general Dirichlet and Neumann problems. Such an extension is outlined in \cite[Section 2.4]{di-pietro.droniou:2020:hybrid}. 
	

Consider a partition of the domain \(\Omega\) into a mesh \(\Mh=( \Th,\Fh)\) where the set of mesh elements \(\Th\) are a set of disjoint polytopes such that \(\ol{\Omega} = \bigcup_{T\in\Th}\ol{T}\) and the set of mesh faces \(\Fh\) form the mesh skeleton \(\bigcup_{T\in\Th}{\bdryT} = \bigcup_{F\in\Fh}\ol{F}\). A detailed definition of this structure can be found in \cite[Definition 1.4]{di-pietro.droniou:2020:hybrid}. The parameter $h$ denotes the maximal element diameter $h\defeq\max_{T\in\Th}\hT$ where, for $X=T\in\Th$ or $X=F\in\Fh$, $h_X$ denotes the diameter of $X$. We shall also collect the set of faces attached to an element $T\in\Th$ in the set $\mathcal{F}_T:=\{F\in\Fh:F\subset T\}$. Similarly, the set containing the one or two elements attached to a face \(F\in\Fh\) is defined as \(\Th[F] \defeq \{T\in\Th:F\subset\bdryT\}\). For each \(T\in\Th\) we denote by \(\norT\) the unit normal directed out of \(T\), and its restriction to a face \(F\in\Fh[T]\) is given by \(\norTF = \norT|_F\). We further make the following assumption on the mesh inline with that stated in \cite{droniou.yemm:2022:robust}.	

\begin{assumption}[Connected by star-shaped sets]\label{assum:star.shaped}	
	There exists a constant \(\varrho>0\) such that for every \(h\in\mathcal{H}\), each \(T\in\Th\) and \(F\in\Fh\) are connected by star-shaped sets with parameter \(\varrho\) (see \cite[Definition 1.41]{di-pietro.droniou:2020:hybrid}). 
\end{assumption}

It is worth noting that Assumption \ref{assum:star.shaped} is independent of the size of and number of faces in each mesh element. Thus, as in \cite{droniou.yemm:2022:robust}, all error estimates in this work remain robust with respect to small and numerous faces.
	
A typical Hybrid High-Order discretisation of problem \eqref{eq:weak.form} relies on piecewise \(H^{k+2}\)-regularity of the solution where \(k\ge0\) is the polynomial degree of the face unknowns. We consider here an exact solution of the form  \( u = u_r + \tilde{u} \) where $u_r$ denotes the `regular part' and $\tilde{u}\in W(\Th)$ where $ W(\Th)$ is a finite dimensional, `weakly singular', enrichment space such that the following assumption holds.

\begin{assumption}[Assumptions on the enrichment space]\label{assum:regularity}
	We assume that every \(\psi\in W(\Th)\) satisfies the following conditions:
	\begin{enumerate}[label=\normalfont(\textbf{A\arabic*}),ref=\normalfont(\textbf{A\arabic*})]
		\item\label{item:A1} \( \forall T\in \Th\), \(\psi \in H^1(T)\),
		\item\label{item:A2} \( \forall T\in \Th\), \(\Delta \psi \in L^2(T)\),
		\item\label{item:A3} \( \forall T\in \Th\), \(\forall F\in \Fh[T]\), \(\nabla \psi \cdot \norTF \in L^2(F)\).
	\end{enumerate}
\end{assumption}

As mentioned, we assume throughout this paper that the exact solution to \eqref{eq:weak.form} can be written as the sum of a regular part and an element of \(W(\Th)\). More specifically, for some \(k\in\N\) we define the solution space to be
\begin{equation}\label{eq:solution.space}	
	\mathcal{V}^{k+2}(\Omega) \defeq \{w\in H^1_0(\Omega): \Delta w \in L^2(\Omega),\, w \in H^{k+2}(\Th) + W(\Th)\},
\end{equation}
where we denote by \(H^{k+2}(\Th)\) the broken Sobolev space
\begin{equation}
	H^{k+2}(\Th) \defeq \{w\in L^{2}(\Omega):w|_T\in H^{k+2}(T)\ \forall T\in\Th\}, \nn
\end{equation}
and assume that the exact solution satisfies $u\in\mathcal{V}^{k+2}(\Omega)$. We note here that Assumption \ref{assum:regularity} does not require the enrichment function to be harmonic. Such a case is considered in Section \ref{sec:oscillatory}.
\begin{remark}
	We note that the conditions on both the regular part and the singular part are purely local and the only global conditions enforced on the exact solution are those of \(u\in \HONEzr(\Omega)\) and \(\Delta u \in \LTWO(\Omega)\). 
\end{remark}

\section{Discrete Problem}

On each $T\in\Th$ we define the discrete space
\begin{equation}\label{def:extended.space}
	\POLYX{k+1}(T) \defeq \POLY{k+1}(T) +  W(T),	
\end{equation}
where \(W(T)\) denotes the restriction to \(T\) of \(W(\Th)\), and \(\POLY{k+1}(T)\) denotes the space of polynomials on \(T\) of degree no more than \(k+1\), \(k\in\N\). The extended elliptic projector on the space \(\POLYX{k+1}(T)\) is defined as the unique \(\piXTe{k+1}:H^1(T)\to\POLYX{k+1}(T)\) such that for all \(v\in\HONE(T)\)
\begin{equation}\label{eq:elliptic.projector}
	\brac[T]{\nabla(v-\piXTe{k+1} v),\nabla w} = 0,\qquad\forall\,w\in\POLYX{k+1}(T)
\end{equation}
and
\begin{equation}\label{eq:elliptic.projector.closure}
	\brac[T]{v-\piXTe{k+1}  v,1} = 0.
\end{equation}
	
Analogous to \cite[Section 2.1]{di-pietro.droniou:2020:hybrid}, we wish to define a discrete space \(\UXTk\), a reconstruction operator \(\pXT{k+1} :\UXTk\to\POLYX{k+1}(T)\), and an interpolator \(\IXTk:H^1(T)\to\UXTk\) such that \(\pXT{k+1} \circ\IXTk = \piXTe{k+1} \). Due to the regularity assumptions on the enrichment space the following integration by parts formula holds for all \(v\in H^1(T)\) and \(w\in\POLYX{k+1}(T)\),
\begin{equation}\label{eq:ibp}
	\brac[T]{\nabla v, \nabla w} = -\brac[T]{v, \Delta w} + \brac[\bdryT]{v, \nabla w \cdot \norT}.
\end{equation}
To introduce projectors to \eqref{eq:ibp} we first define the local discrete spaces
\begin{equation}\label{eq:element.and.face.spaces}
	\POLYD{k}(T) := \POLY{k}(T) + \Delta  W(T) \quad\textrm{and}\quad \POLYN{k}(F) := \POLY{k}(F) + \sum_{T\in\Th[F]}\nabla  W(T)\cdot\norTF.
\end{equation}

\begin{remark}
	A common application of the method designed in this paper is to irregularities in the solution arising from corners in the domain. In such cases, the singular functions considered are harmonic (see Section \ref{sec:numerical}) and the space $\POLYD{k}(T)$ coincides with the polynomial space $\POLY{k}(T)$.
\end{remark}

\begin{remark}
	Much of the time, the gradient of an enrichment function \(\psi \in W(\Th)\) is continuous across each face \(F\in\Fh\) and the definition of \(\POLYN{k}(F)\) is equivalent to 
	\begin{equation}
		\POLYN{k}(F) \defeq \POLY{k}(F) + \nabla  W(\Th)\cdot\norF \nn
	\end{equation}
	for an arbitrary normal \(\norF\) to the face \(F\). However, definition \eqref{eq:element.and.face.spaces} is still well defined for discontinuous enrichment functions. As the space $\POLY{k}(F)$ is defined independent of any particular $T\in\Th$, it is essential in such cases to include the Neumann traces from both elements $T\in\Th[F]$ attached to the face $F$. This is particularly useful when considering locally enriched schemes. Further discussion on this matter is given in Section \ref{sec:local.enrichment}.
\end{remark} 
	
The discrete broken space on an element boundary is defined as 
\[
	\POLYN{k}(\Fh[T]) \defeq \{v\in L^1(\bdryT) : v|_F\in \POLYN{k}(F)\quad\forall F\in\Fh[T]\}.
\] 
It follows from Assumptions \ref{item:A2} and \ref{item:A3} that
\begin{equation}
	\POLY{k}(T) \subset \POLYD{k}(T) \subset L^2(T) \quad\textrm{and}\quad \POLY{k}(\Fh[T]) \subset \POLYN{k}(\Fh[T]) \subset L^2(\bdryT). \nn
\end{equation}	

We denote by \(\piXTzr{k}\) and \(\piXFTzr{k}\) the \(\LTWO\)-orthogonal projectors on \(\POLYD{k}(T)\) and \(\POLYN{k}(\Fh[T])\) respectively. 
Thus, it follows from \eqref{eq:elliptic.projector} and \eqref{eq:ibp}, as well as the inclusions \(\Delta \POLYX{k+1}(T) \subset \POLYD{k}(T)\) and \(\nabla \POLYX{k+1}(T)\cdot\norT \subset \POLYN{k}(\Fh[T])\), that for all \(v\in H^1(T)\),
\begin{equation}\label{eq:elliptic.projector.formula}
	\brac[T]{\nabla \piXTe{k+1} v, \nabla w} = -\brac[T]{\piXTzr{k} v, \Delta w} + \brac[\bdryT]{\piXFTzr{k} v, \nabla w \cdot \norT}\qquad\forall w\in\POLYX{k+1}(T). 
\end{equation}
\subsection{Local Space}
The local space of unknowns is defined to be
\begin{equation}\label{def:local.space}
	\UXTk \defeq \POLYD{k}(T) \times \POLYN{k}(\Fh[T]). 
\end{equation}
For all \( \ulvXT=(\vXT,\vXFT)\in\UXTk \), the reconstruction operator \( \pXT{k+1} :\UXTk\to\POLYX{k+1}(T) \) is defined by
\begin{equation}\label{eq:reconstruction}
	\brac[T]{\nabla \pXT{k+1} \ulvXT, \nabla w} = -\brac[T]{\vXT, \Delta w} + \brac[\bdryT]{\vXFT, \nabla w \cdot \norT} \qquad\forall\,w\in\POLYX{k+1}(T)
\end{equation}
and
\begin{equation}\label{eq:reconstruction.closure}
	\brac[T]{\vXT-\pXT{k+1} \ulvXT,1} = 0.
\end{equation}
Naturally, the interpolator is defined as \( \IXTk v \defeq \brac{\piXTzr{k} v, \piXFTzr{k} v} \). By comparing equations \eqref{eq:elliptic.projector.formula} and \eqref{eq:reconstruction} as well as the closure condition \eqref{eq:reconstruction.closure} we observe the desired commutation property
\begin{equation}\label{eq:commutation}
	\pXT{k+1} \circ\IXTk = \piXTe{k+1} .
\end{equation} 
Due to the \(L^2\)-regularity of the unknown spaces, the local reconstruction and interpolator mimic those defined in \cite{di-pietro.droniou:2020:hybrid}, and the formulation of the discrete problem follows a standard procedure (albeit with non-standard analysis). The continuous form $\a(\cdot,\cdot)$ is approximated on each element by the discrete form $\aXT:\UXTk\times\UXTk\to\R$ defined by
\begin{equation}
	\aXT(\uluXT,\ulvXT) \defeq \brac[T]{\nabla \pXT{k+1} \uluXT,\nabla \pXT{k+1} \ulvXT}+\sXT(\uluXT,\ulvXT), \nn
\end{equation}
where $\sXT:\UXTk\times\UXTk\to\R$ is a stabilisation term satisfying the following assumption. From hereon we shall write \(f\lesssim g\) to mean \(f \le Cg\) where \(C\) is a constant depending only on \(\Omega\), \(k\) and the mesh regularity parameter \(\varrho\).

\begin{assumption}[Local stabilisation term]\label{assum:stability}
	The stabilisation term $\sXT$ is a symmetric, positive semi-definite bilinear form that satisfies:
	\begin{enumerate}
		\item \emph{Coercivity.} For all \(\ulvXT\in\UXTk\) it holds that
		\begin{equation}\label{eq:stab.vol.coercivity}
			\hT^{-2}\norm[T]{\vXT-\pXT{k+1} \ulvXT}^2 \lesssim  \aXT(\ulvXT,\ulvXT) 
		\end{equation}
		and
		\begin{equation}\label{eq:stab.bdry.coercivity}
			\hT^{-1}\norm[\bdryT]{\vXFT-\pXT{k+1} \ulvXT}^2 \lesssim  \aXT(\ulvXT,\ulvXT). 
		\end{equation}
		\item \emph{Consistency.} For all \(w=w_r + \psi\) where \(w_r\in H^{k+2}(T)\) and \(\psi\in W(T)\), it holds that
		\begin{equation}\label{eq:stab.consistency}
			\sXT(\IXTk w, \IXTk w) \lesssim \locerrorRHS{w_r}. 
		\end{equation}
	\end{enumerate}
\end{assumption}

\begin{remark}
	For the case of regular, polynomial unknowns, the two coercivity conditions \eqref{eq:stab.vol.coercivity} and \eqref{eq:stab.bdry.coercivity} are equivalent to the single coercivity condition stated in \cite[Assumption 2]{droniou.yemm:2022:robust}. However, we have to consider the two conditions here to account for the lack of regularity of the solution.
\end{remark}

Some examples of stabilisation terms satisfying Assumption \ref{assum:stability} are given in Section \ref{sec:stab}.

The definition \eqref{eq:elliptic.projector} of the extended elliptic projector and the consistency \eqref{eq:stab.consistency} infer for all \(\psi \in W(T)\) the identities
\begin{equation}\label{eq:psi.invariance}
	\piXTe{k+1} \psi = \psi \qquad\textrm{and}\qquad \sXT(\IXTk\psi, \IXTk\psi) = 0, 
\end{equation}
which, together with \eqref{eq:commutation} and the definition of \(\aXT\), yield
\begin{equation}\label{eq:consistency.on.psi}
	\aXT(\IXTk\psi, \IXTk v) = \a(\psi, v)\qquad\forall v \in H^1(T).
\end{equation}
Equation \eqref{eq:consistency.on.psi} shows the scheme to be exact on elements of \(W(T)\). More generally, \(\psi\) can be replaced by an arbitrary \(w\in\POLYX{k+1}(T)\) in equations \eqref{eq:psi.invariance} and \eqref{eq:consistency.on.psi} which, as in standard HHO, shows consistency on the discrete space.

\subsection{Global Space}
The global space of unknowns is defined as 
\begin{align}\label{def:global.space}
	\UXhkzr\defeq \Big\{\ulvXh=((\vXT)_{T\in\Th},(\vXF)_{F\in\Fh})\,:\,\vXT\in\POLYD{k}(T)\quad\forall T\in\Th, \nl
	\vXF\in\POLYN{k}(F)\quad\forall F\in\Fh\,,\vXF=0\quad\forall F \subset \partial \Omega \Big\}.
\end{align}
For any \(\ulvXh\in\UXhkzr\)  we denote its restriction to an element \(T\) by \(\ulvXT = (\vXT , \vXFT ) \in \UXTk\) (where,
naturally, \(\vXFT\) is defined from \((\vXF)_{F\in\Fh[T]}\)). We also denote by \(\vXh\) the piecewise function satisfying \(\vXh|_T=\vXT\) for all \(T\in\Th\). The global problem reads: find \(\uluXh\in\UXhkzr\) such that
\begin{equation}
	\label{eq:discrete.problem}
	\aXh(\uluXh, \ulvXh) = \brac[\Omega]{f,\vXh}\qquad\forall \ulvXh\in \UXhkzr,
\end{equation}
where 
\begin{equation}
	\aXh(\uluXh, \ulvXh) \defeq \sum_{T\in\Th}\aXT(\uluXT, \ulvXT). \nn
\end{equation}
We endow the global space \(\UXhkzr\) with the norm
\begin{equation}
	\energynorm{\ulvXh} \defeq \aXh(\ulvXh, \ulvXh)^\frac12. \nn
\end{equation}
We shall also denote by \(\IXhk\) and \(\pXh{k+1}\) the global operators whose restrictions to an element \(T\) are given by  \(\IXTk\) and \(\pXT{k+1}\) respectively. The global stabilisation term is defined as
\begin{equation}
	\sXh(\uluXh, \ulvXh) \defeq \sum_{T\in\Th}\sXT(\uluXT, \ulvXT). \nn
\end{equation}

\section{Error Analysis}\label{sec:analysis}

Providing a robust error analysis is often the pitfall of enriched schemes as many tools available for polynomial spaces (such as discrete trace and inverse inequalities) no longer apply. The lack of \(H^1\)-regularity of the element unknowns in \(\POLYD{k}\) means we cannot consider their gradients or traces. We are also restricted to considering approximation properties of the extended elliptic projector only in the \(H^1\)-seminorm. Despite these shortcomings, we are able to show consistency of the scheme and provide estimates for the discrete and continuous energy errors. 

We first make note of the following continuous trace inequality	(c.f. \cite{droniou.yemm:2022:robust}): for all \(T\in\Th\), \(v\in H^1(T)\),
\begin{equation}
\label{eq:trace.continuous}
\hT\norm[\bdryT]{v}^2 \lesssim \norm[T]{v}^2 + \hT^2\seminorm[H^1(T)]{v}^2. 
\end{equation}
The following lemma is also crucial to the consistency of the scheme.

\begin{lemma}[Characterisation of \(\mathcal{V}^{k+2}(\Omega)\)]
	For all \(w\in\mathcal{V}^{k+2}(\Omega)\), \(F\in\Fh\), \(F\not\subset\partial \Omega\),
	\begin{equation}\label{eq:gradient.regularity}
		\nabla w \cdot \bmn_{T_1 F} + \nabla w \cdot \bmn_{T_2 F} = 0,
	\end{equation}
	where \(\{T_1,T_2\} = \Th[F]\).
\end{lemma}

\begin{proof}
	The proof is analogous to that of \cite[Lemma 1.24]{di-pietro.ern:2011:mathematical} once we note that for all \(w\in\mathcal{V}^{k+2}(\Omega)\) we have: \(\Delta w \in L^2(\Omega)\), \(w \in H^1(\Omega)\), and \(\nabla w \cdot \norT \in L^2(\bdryT)\) for all \(T\in\Th\).
\end{proof} 

\begin{theorem}[Consistency Error]\label{thm:consistency.error}
	Let \( w = w_r + \psi \in \mathcal{V}^{k+2}(\Omega)\) with \(w_r\in H^{k+2}(\Th)\) and \(\psi\in  W(\Th)\). The consistency error is given by the linear form \(\Eh (w;\cdot):\UXhkzr\to\R\) defined for all \(\ulvXh\in\UXhkzr\) via
	\begin{equation}
		\Eh (w;\ulvXh) \defeq -\brac[\Omega]{\Delta w, \vXh} - \aXh(\IXhk w,\ulvXh). \nn
	\end{equation}
	The consistency error satisfies
	\begin{equation}\label{eq:error.consistency}
		\big| \Eh (w;\ulvXh) \big| \lesssim \energynorm{\ulvXh} h^{k+1}\seminorm[H^{k+2}(\Th)]{w_r}.
	\end{equation}
\end{theorem}

\begin{proof}
	Consider
	\begin{align}
		-\brac[\Omega]{\Delta w, \vXh} \eq \sum_{T\in\Th}-\brac[T]{\Delta w, \vXT} \nl
		\eq \sum_{T\in\Th}-\brac[T]{\Delta w, \vXT} + \brac[\bdryT]{\nabla w\cdot\norT, \vXFT}, \nn
	\end{align}
	where we justify introducing the term
	\begin{align}
		\sum_{T\in\Th}\brac[\bdryT]{\nabla w\cdot\norT, \vXFT} \eq \sum_{T\in\Th}\sum_{F\in\Fh[T]} \brac[F]{\nabla w\cdot\norTF, \vXF} \nl \eq \sum_{F\in\Fh}\sum_{T\in\Th[F]} \brac[F]{\nabla w\cdot\norTF, \vXF} = 0 \nn
	\end{align}
	due to equation \eqref{eq:gradient.regularity} and the homogeneous condition \eqref{def:global.space} on the discrete space.
	Due to the commutation property \eqref{eq:commutation}, the discrete form is given by
	\begin{align}
		\aXh(\IXhk w,\ulvXh) \eq \sum_{T\in\Th}\Big[\brac[T]{\nabla \piXTe{k+1} w, \nabla \pXT{k+1}  \ulvXT}\Big] + \sXh(\IXhk w,\ulvXh) \nl
		\eq \sum_{T\in\Th}\Big[-\brac[T]{\Delta \piXTe{k+1} w, \vXT} + \brac[\bdryT]{\nabla \piXTe{k+1} w\cdot\norT, \vXFT}\Big] + \sXh(\IXhk w,\ulvXh). \nn
	\end{align}
	Therefore,
	\begin{align}\label{proof.error.consistency.exact}
		\Eh (w;\ulvXh) \plus \sXh(\IXhk w,\ulvXh) \nl \eq \sum_{T\in\Th}\Big[-\brac[T]{\Delta(w- \piXTe{k+1} w), \vXT} + \brac[\bdryT]{\nabla(w- \piXTe{k+1} w)\cdot\norT, \vXFT}\Big] \nl
		\eq \sum_{T\in\Th}\Big[-\brac[T]{\Delta(w_r- \piXTe{k+1} w_r), \vXT} + \brac[\bdryT]{\nabla(w_r- \piXTe{k+1} w_r)\cdot\norT, \vXFT}\Big], 
	\end{align}
	where \(\psi\) has been eliminated via the invariance \eqref{eq:psi.invariance} of \(\piXTe{k+1}\). Let us denote by \(\piTe{k+1}  \) the elliptic projector \cite{di-pietro.droniou:2020:hybrid} on the polynomial space \(\POLY{k+1}(T)\subset \POLYX{k+1}(T)\).
	On each element \(T\in\Th\) we introduce the elliptic projector as follows,
	\begin{align}\label{proof.error.consistency.1}
		- \brac[T]{\Delta({}&w_r - \piXTe{k+1} w_r), \vXT} + \brac[\bdryT]{\nabla(w_r- \piXTe{k+1} w_r)\cdot\norT, \vXFT} \nl
		\eq-\brac[T]{\Delta(w_r- \piTe{k+1}  w_r), \vXT} + \brac[\bdryT]{\nabla(w_r- \piTe{k+1}  w_r)\cdot\norT, \vXFT} \nl \minus \brac[T]{\Delta(\piTe{k+1}  w_r- \piXTe{k+1} w_r), \vXT} + \brac[\bdryT]{\nabla(\piTe{k+1}  w_r- \piXTe{k+1} w_r)\cdot\norT, \vXFT}. 
	\end{align}
	As \(\piTe{k+1}  w_r- \piXTe{k+1} w_r\in\POLYX{k+1}(T)\) we may invoke the definition \eqref{eq:reconstruction} of the potential reconstruction to write
	\begin{align}\label{proof.error.consistency.2}
		- \brac[T]{\Delta(\piTe{k+1}  w_r- \piXTe{k+1} w_r), \vXT} \plus \brac[\bdryT]{\nabla(\piTe{k+1}  w_r- \piXTe{k+1} w_r)\cdot\norT, \vXFT} \nl
		\eq \brac[T]{\nabla(\piTe{k+1}  w_r- \piXTe{k+1} w_r), \nabla \pXT{k+1} \ulvXT}. 
	\end{align}
	Consider also
	\begin{align}\label{proof.error.consistency.3}
		-\brac[T]{\Delta(w_r \,-\, & \piTe{k+1}  w_r), \vXT} + \brac[\bdryT]{\nabla(w_r- \piTe{k+1}  w_r)\cdot\norT, \vXFT} \nl
		\eq -\brac[T]{\Delta(w_r - \piTe{k+1}  w_r), \vXT - \pXT{k+1}\ulvXT} - \brac[T]{\Delta(w_r - \piTe{k+1}  w_r), \pXT{k+1}\ulvXT} \nl \plus \brac[\bdryT]{\nabla(w_r- \piTe{k+1}  w_r)\cdot\norT, \vXFT} \nl
		\eq -\brac[T]{\Delta(w_r - \piTe{k+1}  w_r), \vXT - \pXT{k+1}\ulvXT} + \brac[T]{\nabla(w_r - \piTe{k+1}  w_r), \nabla \pXT{k+1}\ulvXT} \nl \plus \brac[\bdryT]{\nabla(w_r- \piTe{k+1}  w_r)\cdot\norT, \vXFT - \pXT{k+1}\ulvXT}, 
	\end{align}
	where we have introduced the term \(\pXT{k+1}\ulvXT\) and integrated by parts. 
	Combining \eqref{proof.error.consistency.1}, \eqref{proof.error.consistency.2} and \eqref{proof.error.consistency.3} yields
	\begin{align}\label{proof.error.consistency.4}
		- \brac[T]{\Delta(w_r \,-\, & \piXTe{k+1} w_r), \vXT} + \brac[\bdryT]{\nabla(w_r- \piXTe{k+1} w_r)\cdot\norT, \vXFT} \nl
		\eq -\brac[T]{\Delta(w_r- \piTe{k+1}  w_r), \vXT-\pXT{k+1} \ulvXT} + \brac[T]{\nabla(w_r- \piXTe{k+1} w_r), \nabla \pXT{k+1} \ulvXT} \nl \plus \brac[\bdryT]{\nabla(w_r- \piTe{k+1}  w_r)\cdot\norT, \vXFT-\pXT{k+1} \ulvXT}. 
	\end{align}
	By the definition \eqref{eq:elliptic.projector} of the extended elliptic projector we have that
	\begin{equation}\label{proof.error.consistency.5}
		\brac[T]{\nabla(w_r- \piXTe{k+1} w_r), \nabla \pXT{k+1} \ulvXT} = 0. 
	\end{equation}
	Therefore, combining \eqref{proof.error.consistency.exact}, \eqref{proof.error.consistency.4} and \eqref{proof.error.consistency.5} yields
	\begin{align}
		\Eh (w;\ulvXh) \eq -\sXh(\IXhk w,\ulvXh) + \sum_{T\in\Th}\Big[-\brac[T]{\Delta(w_r- \piTe{k+1}  w_r), \vXT-\pXT{k+1} \ulvXT} \nl \plus \brac[\bdryT]{\nabla(w_r- \piTe{k+1}  w_r)\cdot\norT, \vXFT-\pXT{k+1} \ulvXT}\Big], \nn
	\end{align}
	and thus
	\begin{align}\label{proof.error.consistency.6}
		\big|\Eh (w;\ulvXh)\big| \lea \sum_{T\in\Th}\big|\sXT(\IXTk w,\ulvXT)\big| + \sum_{T\in\Th}\big|\brac[T]{\Delta(w_r- \piTe{k+1}  w_r), \vXT-\pXT{k+1} \ulvXT}\big| \nl \plus \sum_{T\in\Th}\big|\brac[\bdryT]{\nabla(w_r- \piTe{k+1}  w_r)\cdot\norT, \vXFT-\pXT{k+1} \ulvXT}\big|.
	\end{align}
	By a Cauchy--Schwarz inequality, the coercivity condition \eqref{eq:stab.vol.coercivity}, and the approximation properties of the elliptic projector \cite[Theorem 1.48]{di-pietro.droniou:2020:hybrid},
	\begin{align}\label{proof.error.consistency.7}
		\big|\brac[T]{\Delta(w_r- \piTe{k+1}  w_r), \vXT-\pXT{k+1} \ulvXT}\big| \lea \norm[T]{\Delta(w_r- \piTe{k+1}  w_r)}\norm[T]{\vXT-\pXT{k+1} \ulvXT}\nl
		\les \seminorm[H^2(T)]{w_r- \piTe{k+1}  w_r}\hT\aXT(\ulvXT,\ulvXT)^\frac12 \nl
		\les \hT^{k+1}\seminorm[H^{k+2}(T)]{w_r}\aXT(\ulvXT,\ulvXT)^\frac12.
	\end{align}
	Similarly, by a Cauchy--Schwarz inequality, the continuous trace inequality \eqref{eq:trace.continuous}, and \eqref{eq:stab.bdry.coercivity},
	\begin{align}\label{proof.error.consistency.8}
		\big|\brac[\bdryT]{\nabla(w_r \minus \piTe{k+1}  w_r)\cdot\norT, \vXFT-\pXT{k+1} \ulvXT}\big| \nl \lea \norm[\bdryT]{\nabla(w_r- \piTe{k+1}  w_r)}\norm[\bdryT]{\vXFT-\pXT{k+1} \ulvXT} \nl
		\les \Brac{\hT^{-\frac12}\seminorm[H^1(T)]{w_r- \piTe{k+1}  w_r} + \hT^\frac12\seminorm[H^2(T)]{w_r- \piTe{k+1}  w_r}} \hT^\frac12\aXT(\ulvXT,\ulvXT)^\frac12\nl
		\les \hT^{k+1}\seminorm[H^{k+2}(T)]{w_r}\aXT(\ulvXT,\ulvXT)^\frac12,
	\end{align}
	where we have again invoked the approximation properties of the elliptic projector. The stabilisation term is bounded using a Cauchy--Schwarz inequality and the consistency condition \eqref{eq:stab.consistency},
	\begin{align}\label{proof.error.consistency.9} 
		\big|\sXT(\IXTk w,\ulvXT)\big| \lea \sXT(\IXTk w,\IXTk w)^\frac12\sXT(\ulvXT,\ulvXT)^\frac12 \nl \les \hT^{k+1}\seminorm[H^{k+2}(T)]{w_r}\aXT(\ulvXT,\ulvXT)^\frac12.
	\end{align}
	Substituting \eqref{proof.error.consistency.7}, \eqref{proof.error.consistency.8}, and \eqref{proof.error.consistency.9} into \eqref{proof.error.consistency.6} yields
	\begin{equation}
		\big|\Eh (w;\ulvXh)\big| \lesssim \sum_{T\in\Th}\hT^{k+1}\seminorm[H^{k+2}(T)]{w_r}\aXT(\ulvXT,\ulvXT)^\frac12. \nn
	\end{equation}
	The proof follows from a discrete Cauchy--Schwarz inequality and noting that \(\hT \le h\) for all \(T\in\Th\).
\end{proof}

The proof of the consistency error given above is significantly more detailed than that of regular HHO given in \cite[Section 2.2]{di-pietro.droniou:2020:hybrid}. The reason for this is partly due the limited regularity of element unknowns, \(\vXT\in L^2(T)\). However, even if we were to assume \(\HONE\)-regularity to arrive at the equation
\begin{multline*}
	\Eh (w;\ulvXh) + \sXh(\IXhk w,\ulvXh) = \sum_{T\in\Th}\Big[\brac[\T]{\nabla(w_r - \piXTe{k+1} w_r), \nabla\vXT} \\ 
	+ \brac[\bdryT]{\nabla(w_r - \piXTe{k+1} w_r)\cdot\norT, \vXFT - \vXT}\Big],
\end{multline*}
this is still not useful. The first term does not equate to zero as $\vXT \notin \POLYX{k+1}(T)$ in general. Moreover, there is no guarantee that \(\norm[\bdryT]{\nabla(w_r - \piXTe{k+1} w_r)\cdot\norT}\) will scale appropriately. Thus, the extra and lengthy details in the proof of Theorem \ref{thm:consistency.error} appear necessary.

\begin{theorem}[Energy error]\label{thm:energy.error}
	Let \(u=u_r+\tilde{u} \in \mathcal{V}^{k+2}(\Omega)\) be the exact solution to the continuous problem \eqref{eq:weak.form} where \(u_r\in H^{k+2}(\Th)\) and \(\tilde{u}\in W(\Th)\). Let $\uluXh\in\UXhkzr$ be the solution to the discrete problem \eqref{eq:discrete.problem}. The following energy error estimates hold:
	\begin{equation} \label{eq:energy.error}
		\energynorm{\uluXh - \IXhk u} + \seminorm[H^1(\Th)]{\pXh{k+1} \uluXh - u} \lesssim h^{k+1}\seminorm[H^{k+2}(\Th)]{u_r}.
	\end{equation}
\end{theorem}

\begin{proof}
	By the coercivity conditions \eqref{eq:stab.bdry.coercivity} and \eqref{eq:stab.vol.coercivity}, as well as the homogeneous conditions on the discrete space, it is clear that $\energynorm{\cdot}$ describes a norm on $\UXhkzr$. As such, we infer from the Third Strang Lemma \cite{di-pietro.droniou:2018:third} that
	\begin{equation}
		\energynorm{\uluXh - \IXhk u} \le \sup_{\ulvXh\ne0}\frac{|\mathcal{E}_h( u;\ulvXh)|}{\energynorm{\ulvXh}}. \nn
	\end{equation}
	Combining with the consistency error \eqref{eq:error.consistency} yields the estimate
	\begin{equation}\label{eq:energy.error.discrete}
		\energynorm{\uluXh - \IXhk u} \lesssim h^{k+1}\seminorm[H^{k+2}(\Th)]{u_r}.
	\end{equation}
	Consider on each element $T\in\Th$ the triangle inequality,
	\begin{align}
		\seminorm[H^1(T)]{\pXT{k+1}  \uluXT - u} \lea
		\seminorm[H^1(T)]{\pXT{k+1}  \uluXT - \piXTe{k+1}  u} + \seminorm[H^1(T)]{\piXTe{k+1}  u - u} \nl
		\eq \seminorm[H^1(T)]{\pXT{k+1}  (\uluXT - \IXTk u)} + \seminorm[H^1(T)]{\piXTe{k+1}  u_r - u_r}\nl 
		\lea \aXT(\uluXT - \IXTk u, \uluXT - \IXTk u)^\frac12 + \seminorm[H^1(T)]{\piTe{k+1}  u_r - u_r},\nn
	\end{align}
	where \(\piXTe{k+1}\) is replaced by \(\piTe{k+1}\) due to orthogonal projectors minimising their respective norms and \(\piTe{k+1}  u_r\in\POLYX{k+1}(T)\).
	Squaring, summing over all \(T\in\Th\) and invoking the approximation properties of the elliptic projector yields
	\begin{equation}
		\seminorm[H^1(\Th)]{\pXh{k+1} \uluXh - u}^2 \lesssim \energynorm{\uluXh - \IXhk u}^2 + \big[h^{k+1}\seminorm[H^{k+2}(\Th)]{u_r}\big]^2. \nn
	\end{equation}
	The proof is complete by applying the estimate \eqref{eq:energy.error.discrete}.
\end{proof}	

\section{Stabilisation}\label{sec:stab}

In this section, we give an example of a stabilisation term satisfying the coercivity and consistency properties of Assumption \ref{assum:stability}. We first give below an extension of \cite[Lemma 2.11]{di-pietro.droniou:2020:hybrid} to the extended discrete spaces considered here. The difference operators \(\deltaXT{k}:\UXTk\to\POLYD{k}(T)\) and \(\deltaXFT{k}:\UXTk\to\POLYN{k}(\Fh[T])\) are defined via
\begin{equation*}
	\deltaXT{k} \ulvXT \defeq \vXT-\piXTzr{k} \pXT{k+1} \ulvXT\qquad\textrm{and}\qquad \deltaXFT{k} \ulvXT \defeq \vXFT-\piXFTzr{k} \pXT{k+1} \ulvXT.
\end{equation*}

\begin{lemma}[Dependency of \(\sXT\)]
	For \(\sXT\) to satisfy the consistency condition \eqref{eq:stab.consistency} it is necessary (but not sufficient) that \(\sXT\) depends on its arguments only through $\deltaXT{k}$ and $\deltaXFT{k}$.
\end{lemma}

\begin{proof}
	Consider a symmetric, positive semi-definite bilinear form \(\sXT\) that satisfies condition \eqref{eq:stab.consistency}. Thus, for all \(w=w_r+\psi\in \POLYX{k+1}(T)\) with \(w_r\in\POLY{k+1}(T)\) and \(\psi\in W(T)\),
	\begin{equation}
		\sXT(\IXTk w, \IXTk w) \lesssim \locerrorRHS{w_r} =  0. \nn
	\end{equation}
	Therefore, by a Cauchy--Schwarz inequality,
	\begin{equation}
		\sXT(\IXTk w, \ulvXT) = 0 \nn
	\end{equation}
	for all \(\ulvXT\in\UXTk\). Thus,
	\begin{align}
	\sXT(\uluXT, \ulvXT) \eq \sXT(\uluXT, \ulvXT - \IXTk\pXT{k+1} \ulvXT) \nl 
	\eq \sXT(\uluXT - \IXTk\pXT{k+1} \uluXT, \ulvXT - \IXTk\pXT{k+1} \ulvXT). \nn
	\end{align}
	The proof is complete by noting that
	\begin{equation}
	\ulvXT-\IXTk\pXT{k+1} \ulvXT = \brac{\deltaXT{k} \ulvXT, \deltaXFT{k} \ulvXT}. \nn
	\end{equation}
\end{proof}

A stabilisation term satisfying Assumption \ref{assum:stability} is obtained setting
\begin{equation}
	\label{def:vol.stab}
	\sXT(\uluXT, \ulvXT) \defeq \hT^{-2}\brac[T]{\deltaXT{k} \uluXT, \deltaXT{k} \ulvXT} + \hT^{-1}\brac[\bdryT]{\deltaXFT{k} \uluXT, \deltaXFT{k} \ulvXT}.
\end{equation}

\begin{lemma}[Coercivity]
	The stabilisation term defined by \eqref{def:vol.stab} satisfies the coercivity conditions \eqref{eq:stab.vol.coercivity} and \eqref{eq:stab.bdry.coercivity}.
\end{lemma}

\begin{proof}
	Consider by a triangle inequality, and the inclusion \(\POLY{k}(T)\subset\POLY{k}_\Delta(T)\) along with the minimisation of \(\piXTzr{k}\) on \(\POLYD{k}(T)\),
	\begin{align}
		\hT^{-2}\norm[T]{\vXT-\pXT{k+1} \ulvXT}^2 \les \hT^{-2}\norm[T]{\vXT-\piXTzr{k} \pXT{k+1} \ulvXT}^2 + \hT^{-2}\norm[T]{\piXTzr{k} \pXT{k+1} \ulvXT-\pXT{k+1} \ulvXT}^2 \nl
		\lea \sXT(\ulvXT, \ulvXT) + \hT^{-2}\norm[T]{\piTzr{k}  \pXT{k+1} \ulvXT-\pXT{k+1} \ulvXT}^2 \nl
		\les \sXT(\ulvXT, \ulvXT) + \seminorm[H^1(T)]{\pXT{k+1} \ulvXT}^2 = \aXT(\ulvXT, \ulvXT), \nn
	\end{align}
	where the final inequality follows from the approximation properties of the \(L^2\)-orthogonal projector on \(\POLY{k}(T)\) \cite[Theorem 1.4.5]{di-pietro.droniou:2020:hybrid}.
	Similarly, by a triangle inequality and noting that \(\POLY{k}(T)|_{\bdryT} \subset \POLYN{k}(\Fh[T])\) (so to replace \(\piXFTzr{k}\) with \(\piTzr{k}\)),
	\begin{align}
		\hT^{-1}\norm[\bdryT]{&\vXFT-\pXT{k+1} \ulvXT}^2 \nl \les \hT^{-1}\norm[\bdryT]{\vXFT-\piXFTzr{k} \pXT{k+1} \ulvXT}^2 + \hT^{-1}\norm[\bdryT]{\piXFTzr{k} \pXT{k+1} \ulvXT-\pXT{k+1} \ulvXT}^2 \nl
		\lea \sXT(\ulvXT, \ulvXT) + \hT^{-1}\norm[\bdryT]{\piTzr{k}  \pXT{k+1} \ulvXT-\pXT{k+1} \ulvXT}^2 \nl
		\les \sXT(\ulvXT, \ulvXT) + \hT^{-2}\norm[T]{\piTzr{k}  \pXT{k+1} \ulvXT-\pXT{k+1} \ulvXT}^2 + \seminorm[H^1(T)]{\piTzr{k}  \pXT{k+1} \ulvXT-\pXT{k+1} \ulvXT}^2 \nl
		\les \sXT(\ulvXT, \ulvXT) + \seminorm[H^1(T)]{\pXT{k+1} \ulvXT}^2 = \aXT(\ulvXT, \ulvXT), \nn
	\end{align}
	where we have used the continuous trace inequality \eqref{eq:trace.continuous} and again invoked the approximation properties of the \(L^2\)-orthogonal projector.
\end{proof}

\begin{lemma}[Consistency]
	The stabilisation term defined by \eqref{def:vol.stab} satisfies the consistency condition \eqref{eq:stab.consistency}.
\end{lemma}

\begin{proof}
	By the definition of \(\sXT\), the commutation property \eqref{eq:commutation}, and the boundedness of orthogonal projectors, it holds that
	\begin{align}
		\sXT(\IXTk w, \IXTk w) \eq \hT^{-2}\norm[T]{\piXTzr{k} w-\piXTzr{k} \piXTe{k+1} w}^2 + \hT^{-1}\norm[\bdryT]{\piXFTzr{k} w-\piXFTzr{k} \piXTe{k+1} w}^2 \nl			
		\lea \hT^{-2}\norm[T]{w-\piXTe{k+1} w}^2 + \hT^{-1}\norm[\bdryT]{w-\piXTe{k+1} w}^2. \nn
	\end{align}
	By the continuous trace inequality \eqref{eq:trace.continuous} and a Poincar\'{e} inequality due to the zero mean value of \(w-\piXTe{k+1} w\), we infer that
	\begin{align}
			\sXT(\IXTk w, \IXTk w) \les \hT^{-2}\norm[T]{w-\piXTe{k+1} w}^2 + \seminorm[H^1(T)]{w-\piXTe{k+1} w}^2 \nl
			\les \seminorm[H^1(T)]{w-\piXTe{k+1} w}^2. \nn
	\end{align}
	We note the inclusion \(\POLY{k+1}(T)\subset \POLYX{k+1}(T)\) and invoke the invariance and minimisation properties of orthogonal projectors to conclude that
	\begin{equation}
		\seminorm[H^1(T)]{w-\piXTe{k+1} w}^2 = \seminorm[H^1(T)]{w_r-\piXTe{k+1} w_r}^2 \le \seminorm[H^1(T)]{w_r-\piTe{k+1}  w_r}^2. \nn
	\end{equation}
	The proof then follows from the approximation properties of the elliptic projector \cite[Theorem 1.48]{di-pietro.droniou:2020:hybrid}.
\end{proof}

Designing alternate stabilisation terms proves difficult due to the limited regularity of the unknowns. However, if we assume that \(\Delta W(T) \subset \POLY{k}(T)\) (so that the element unknowns are polynomials), the stabilisation terms
\begin{align}
	\sXT^\nabla(\uluXT, \ulvXT) \deq \brac[T]{\nabla \deltaXT{k}\uluXT, \nabla \deltaXT{k}\ulvXT} + \hT^{-1}\brac[\bdryT]{\deltaXFT{k}\uluXT, \deltaXFT{k}\ulvXT} , \label{eq:stab.grad}\\
	\sXT^\partial(\uluXT, \ulvXT) \deq \hT^{-1}\brac[\bdryT]{(\deltaXFT{k}-\deltaXT{k})\uluXT, (\deltaXFT{k}-\deltaXT{k})\ulvXT} \label{eq:stab.bdry}
\end{align}
as defined in \cite[Section 4]{droniou.yemm:2022:robust} both satisfy Assumption \ref{assum:stability}. While this may at first seem contrived, it is quite natural to consider a singular enrichment function with zero Laplacian (see Section \ref{sec:numerical}) which clearly satisfies the aforementioned condition. 

\section{Implementation and Numerical Tests}\label{sec:numerical}

As mentioned in the introduction, singular solutions to \eqref{eq:weak.form} can arise from corners in the boundary of an otherwise smooth domain. In particular, at non-convex corners we cannot even assume \(H^2\)-regularity of the solution \cite[c.f.][]{grisvard:1985:elliptic}. Moreover, singularities on the boundary can arise from irregular boundary data, or a transition from Dirichlet to Neumann data. For simplicity, we consider here a domain \(\Omega\subset \R^2\) with one re-entrant corner located at the origin. Thus, after a possible rotation of coordinates, the domain \(\Omega\) corresponds with the region \(\{(x_1, x_2) = (r \cos \theta, r \sin \theta)|\,r > 0, 0 < \theta < \omega\}\) in some neighbourhood of the origin, where \(\pi < \omega \le 2\pi\) is the angle of the re-entrant corner. For each \(j\in\N\) we define a function
\begin{equation}\label{eq:singular.functions}
	u_j = 
	\begin{cases}
		r^\frac{j\pi}{\omega}\sin(\frac{j\pi}{\omega} \theta)&\qquad\textrm{if } \frac{j \pi}{\omega} \not\in \Z \\
		r^\frac{j\pi}{\omega}(\ln r\sin(\frac{j\pi}{\omega}\theta)+\theta \cos(\frac{j\pi}{\omega}\theta))&\qquad\textrm{if } \frac{j \pi}{\omega} \in \Z
	\end{cases}.
\end{equation}
Given a source term \(f\in H^k(\Omega)\), there exist numbers \(c_j\) such that the solution to the homogeneous Dirichlet problem \eqref{eq:weak.form} satisfies
\begin{equation}\label{eq:regularity.result}
	u - \sum_{1 \le j < \frac{\omega}{\pi}(k+1)} c_j u_j \in H^{k+2}(\Omega_0), 
\end{equation}
where \(\Omega_0\subset \Omega\) is some open neighbourhood of the origin. We refer the reader to \cite[Chapter 5]{grisvard:1985:elliptic}, which is dedicated to proving \eqref{eq:regularity.result} and equivalent results on polygonal domains with generic boundary data. Each singular function \(u_j\) clearly satisfy Assumptions \ref{item:A1} and \ref{item:A2} (the latter due to \(\Delta u_j = 0\)). On an edge \(F\in\Fh\) containing the singular point \(r=0\), the least regular function (\(j=1\)) satisfies \(\nabla u_1 \cdot \norF \in L^p(F)\) for all \(p < \frac{\omega}{\omega-\pi}\). Therefore, Assumption \ref{item:A3} holds true provided that \(\omega < 2\pi\) (corresponding to a cracked domain). 

In practice, to assure \(H^{k+2}\)-regularity in a polygonal domain, singular functions require to be defined at every corner. However, for computational simplicity, we consider only one singular function defined at the re-entrant corner. In particular, we consider here an L-shape domain \( \Omega = (-1,1)^2\,\backslash\,[0,1]^2 \) and exact solution
\begin{equation}
	u = \sin(\pi x_1) \sin(\pi x_2) + \psi, \nn
\end{equation}
where \(\psi = r^\alpha \sin(\alpha (\theta - \frac{\pi}{2}))\), \(\alpha = \frac{\pi}{\omega} = \frac23\) and \(\frac{\pi}{2} \le \theta \le 2 \pi\). Naturally, we define the enrichment space as \(W(\Th) = \textrm{span}\{\psi\}\). We remark that the exact solution considered does not have homogeneous boundary conditions. However, as mentioned in the introduction, the extension of the XHHO scheme to inhomogeneous boundary data follows seamlessly. Indeed, if we consider the inhomogeneous condition $u=g_D$ on $\partial \Omega$, we define $\ul{u}_{\xh,D}=\brac{\brac[T\in\Th]{0}, \brac[F\in\Th]{u_{\xF,D}}}\in\UXhk$ where
\[
	u_{\xF,D} = 
	\begin{cases}
	\piXFzr{k}g_D&\ \textrm{if } F\subset\partial \Omega \\
	0&\ \textrm{otherwise}
	\end{cases}. \nn
\] 
The inhomogeneous problem is given by: find $\ul{u}_{\xh,0}\in\UXhkzr$ such that
\[
	\aXh(\ul{u}_{\xh,0}, \ulvXh) = \brac[\Omega]{f, \vXh} - \aXh(\ul{u}_{\xh,D}, \ulvXh) \qquad \forall\ulvXh\in\UXhkzr,
\]
and the discrete solution is given by $\uluXh = \ul{u}_{\xh,0} + \ul{u}_{\xh,D}$.

\subsection{Basis Functions}

The Extended Hybrid High-Order scheme designed in this paper requires the computation of spaces \(\POLYX{k+1}(T)\) and \(\POLYD{k}(T)\) on each element \(T\in\Th\), and of the space \(\POLYN{k}(F)\) on each face \(F\in\Fh\). As the Laplacian of the singular function is \(0\), i.e. \(\Delta W(\Th) = \{0\}\), it holds that \(\POLYD{k}(T) = \POLY{k}(T) \subset \POLYX{k+1}(T)\). Therefore, on each mesh element we define basis functions for the space \(\POLYX{k+1}(T)\) and consider \(\POLYD{k}(T)\) a subspace.

Consider monomial basis functions, \(\{\phi_T^j\}_{j=1}^{\textrm{dim}\{\POLY{k+1}(T)\}}\) and \(\{\phi_F^j\}_{j=1}^{\textrm{dim}\{\POLY{k}(F)\}}\), in locally scaled coordinates for the respective polynomial spaces. On each mesh element we write the extended space as
\begin{equation}\label{eq:cell.basis}
	\POLYX{k+1}(T) = \textrm{span}\big\{ (\phi_T^j)_{j=1}^{\textrm{dim}\{\POLY{k+1}(T)\}}, \psi \big\}, 
\end{equation}
and we then \(L^2\)-orthonormalise the basis functions following a Gram-Schmidt process. It is natural to define a basis for the enriched space on the mesh faces as
\begin{equation}\label{eq:face.basis}
	\POLYN{k}(F) = \textrm{span}\big\{ (\phi_F^j)_{j=1}^{\textrm{dim}\{\POLY{k}(F)\}}, \nabla \psi \cdot \norF \big\}, 
\end{equation} 
for some choice of normal vector \(\norF\). However, there may exist certain faces \(F\in\Fh\) such that this choice of basis is not linearly independent. Consider a unit normal to a face \(F\in\Fh\) defined by \(\norF = (\cos \theta_\nor, \sin \theta_\nor)\) for some constant angle \(\theta_\nor\). Taking the scalar product with \(\nabla \psi\),
\begin{equation}\label{eq:grad.dot.n}
	\nabla \psi \cdot \norF = \alpha r^{\alpha - 1} \sin(\theta_\nor - \theta + \alpha(\theta - \frac{\pi}{2})).
\end{equation}
If \(\theta \equiv const\) then \(\theta_{\nor} = \theta + (2 j + 1) \frac{\pi}{2} \), for an integer \(j \in \N \). Substituting into \eqref{eq:grad.dot.n} yields 
\begin{equation}
	\nabla \psi \cdot \norF =  \pm\alpha r^{\alpha - 1} \cos(\alpha(\theta - \frac{\pi}{2})) \nn
\end{equation}
which is identically \(0\) if \(\theta \equiv \frac{(6j-1)\pi}{4}\) for some \(j\in\N\). Thus, as we only consider \(\frac{\pi}{2} \le \theta \le 2\pi\), the basis for \(\POLYN{k}(F)\) defined by \eqref{eq:face.basis} is linearly independent except along faces such that \(\theta = const = \frac{5\pi}{4}\) (for which \(\POLYN{k}(F) = \POLY{k}(F)\)). The bases on each face are then also \(L^2\)-orthonormalised via a Gram-Schmidt process.

\subsection{Integration Rules}

It is well known that classical Gauss-Legendre quadrature rules suffer from large errors and numerical instability near a singularity. To design an alternate numerical integration rule we make use of the fact that \(\psi\) is a homogeneous function of degree \(\alpha\). Thus, we can use a homogeneous integration rule (as described in \cite{chin.lasserre.ea:2015:numerical}) to rewrite the volumetric integrals on the boundary. In particular we consider the integral \( \int_T \psi \phi^j_T \)
for some \(1 \le j \le \textrm{dim}\{\POLY{k+1}(T)\}\). As \(\phi_T^j\) is a monomial in local coordinates \(\tilde{\bmx} = \frac{\bmx - \bmx_0}{\hT}\), it is homogeneous in \(\bmx - \bmx_0\). By Euler's homogeneous function theorem, \(\nabla \psi \cdot \bmx = \alpha \psi\) and \(\nabla \phi_T^j \cdot (\bmx-\bmx_0) = q \phi_T^j\), where \(q\) is the degree of the monomial \(\phi_T^j\). Therefore,
\begin{align}
	\int_{\bdryT} \psi \phi^j_T \bmx\cdot\norT \eq \int_T \psi \phi^j_T \nabla \cdot \bmx + \int_T \bmx \cdot\nabla (\psi \phi^j_T) \nl
	\eq (d + \alpha + q)\int_T \psi \phi^j_T + \int_T \psi\bmx_0 \cdot\nabla \phi^j_T, \nn
\end{align}
which yields
\begin{equation}
	(d + \alpha + q)\int_T \psi \phi^j_T = \sum_{F\in\Fh[T]}\bmx_F\cdot\norTF\int_F\psi \phi^j_T - \int_T \psi\bmx_0 \cdot\nabla \phi^j_T, \nn
\end{equation}
where \(\bmx_F\) is an arbitrary point in the face \(F\). As each component of \(\nabla \phi^j_T\) is a monomial of degree \(q- 1\), the above process can be repeated iteratively until the whole integral is described on the boundary. For edges that do not pass through the singular point \(r=0\), the integral \(\int_F\psi \phi^j_T\) can be accurately approximated using Gauss-Legendre quadrature. On edges \(F\in\Fh\) with \(\theta \equiv const\) (which include edges passing through the origin) the integral can be computed exactly. Consider the arc-length parametrisation \(\bmx = (r \cos(\theta), r \sin(\theta))\), \(R_0 \le r \le R_1\) and write
\begin{equation}
	\int_F\psi \phi^j_T = \int_{R_1}^{R_2} \psi(\bmx(r)) \phi_T^j(\bmx(r))\,dr. \nn
\end{equation} 
An integration by parts yields
\begin{align}
	r\psi(\bmx(r)) \phi_T^j(\bmx(r))\Big|_{r=R_1}^{R_2} \eq \int_{R_1}^{R_2} \psi  \phi_T^j \frac{dr}{dr}\,dr + \int_{R_1}^{R_2} r\frac{d}{dr}(\psi \phi_T^j)\,dr \nl
	\eq \int_{R_1}^{R_2} \psi \phi_T^j\,dr + \int_{R_1}^{R_2} r(\phi_T^j \nabla \psi \cdot \frac{d\bmx}{dr} + \psi \nabla \phi_T^j \cdot \frac{d\bmx}{dr})\,dr. \label{eq:integration.on.faces}
\end{align}
Noting that \(\bmx = r \frac{d\bmx}{dr}\), and again invoking the homogeneity results \(\nabla \psi \cdot \bmx = \alpha \psi\) and \(\nabla \phi_T^j \cdot (\bmx - \bmx_0) = q \phi_T^j\), equation \eqref{eq:integration.on.faces} is evaluated as
\begin{equation}
	\norm{\bmx}\psi(\bmx) \phi_T^j(\bmx)\Big|_{\bmx=\bmv_0}^{\bmv_1} = (1 + \alpha + q)\int_{R_1}^{R_2} \psi \phi_T^j\,dr + \int_{R_1}^{R_2} \psi \bmx_0 \cdot \nabla \phi_T^j\,dr, \nn
\end{equation}
where \(\bmv_0\) and \(\bmv_1\) are the vertices corresponding to \(r=R_0\) and \(r=R_1\) respectively. Again, we may follow an iterative procedure for each component of \(\int_{R_1}^{R_2} \psi \bmx_0 \cdot \nabla \phi_T^j\,dr\) to evaluate the entire integral \(\int_F\psi \phi^j_T\) on the vertices \(\bmv_0\), \(\bmv_1\). 

The calculations of integrals involving gradients is marginally simpler as they can be transformed directly onto the boundary via
\begin{equation}
	\int_T \nabla \psi \cdot \nabla \phi_T^j = \int_{\bdryT} \phi_T^j \nabla \psi \cdot \norT, \nn
\end{equation}
(due to \(\Delta \psi = 0\)). Each edge integral is calculated using the process described above, noting that each \(\nabla \psi \cdot \norTF\) is a homogeneous function of degree \(\alpha - 1\).

\subsection{Local Enrichment}\label{sec:local.enrichment}

Far from the singular point \(r=0\), the enrichment function \(\psi\) is smooth. More precisely, for any open set \(V\) containing the point \(r=0\), \(\psi \in C^\infty(\Omega\backslash V)\). As such, away from the singularity, \(\psi\) can be efficiently approximated by a polynomial. This can cause the local matrices appearing in the orthonormalisation process of the bases \eqref{eq:cell.basis} and \eqref{eq:face.basis} to be highly ill conditioned. Naturally, this leads us to consider a scheme such that the unknowns spaces are enriched near the singular point and are otherwise polynomial spaces. We follow a similar approach to that taken in \cite[Section 3.2]{artioli.mascotto:2021:enrichment}.

Consider a parameter \(\gamma > 0 \) and define the set
\begin{equation}
	\label{eq:enriched.elements}
	\mathcal{T}_{\gamma} \defeq \{ T \in \Th: \norm{{\bm{0}} - \bmx_T} < \gamma \},
\end{equation}
where we denote by \(\bmx_T\) the centroid of an element \(T\). A cut-off function \(\delta_{\gamma}:\Omega\to\R\) is defined by
\begin{equation}
	\delta_{\gamma}|_T = 
	\begin{cases}
	1&\ \textrm{if } T \in \mathcal{T}_{\gamma} \\
	0&\ \textrm{otherwise}
	\end{cases}. \nn
\end{equation}
If \(u=u_r + \psi\), with \(u_r \in C^\infty(\Omega)\) we can write \(u = u_r + \psi - \delta_{\gamma}\psi + \delta_{\gamma}\psi = \tilde{u}_r + \delta_{\gamma}\psi\) where \(u_r + \psi - \delta_{\gamma}\psi = \tilde{u}_r \in H^{k+2}(\Th)\). Thus, we define the enrichment space as \(W(\Th) = \textrm{span}\{\delta_{\gamma}\psi\}\). As the assumptions on the enrichment space are all made locally, \(W(\Th)\) clearly still satisfies Assumption \ref{assum:regularity}. Definitions \eqref{def:extended.space} and \eqref{eq:element.and.face.spaces} of the discrete spaces then correspond to
\begin{align}
	\POLYX{k+1}(T) \eq
	\begin{cases}
	\POLY{k+1}(T) + \textrm{span}\{\psi\}&\ \textrm{if } T \in \mathcal{T}_{\gamma} \\
	\POLY{k+1}(T)&\ \textrm{otherwise}
	\end{cases}\nl
	\POLYN{k}(F) \eq
	\begin{cases}
	\POLY{k}(F) + \textrm{span}\{\nabla\psi\cdot\norF\}&\ \textrm{if } F \in \mathcal{F}_{\gamma} \\
	\POLY{k}(F)&\ \textrm{otherwise}
	\end{cases} \nn
\end{align}
for an arbitrary normal \(\norF\), where
\begin{equation}
	\mathcal{F}_{\gamma} = \{ F \in \Fh: \Th[F]\cap\mathcal{T}_\gamma \ne \emptyset \}. \nn
\end{equation}

\subsection{Tests}

As mentioned previously, we consider an  L-shape domain \( \Omega = (-1,1)^2\,\backslash\,[0,1]^2 \) and exact solution \(u = \sin(\pi x_1) \sin(\pi x_2) + r^\alpha \sin(\alpha( \theta - \frac{\pi}{2}))\), $\alpha=\frac23$. The stabilisation term defined by \eqref{def:vol.stab} is considered. As the element unknowns are polynomials, more choices of stabilisation terms exist (see Section \ref{sec:stab}). A comparison between the various stabilisation terms is carried out for the standard HHO method in \cite[Section 5]{droniou.yemm:2022:robust}. The XHHO scheme is tested on a sequence of hexagonal meshes with maximum element diameter \(h\to 0\). Two members of this family are plotted in Figure \ref{fig:mesh.Lshape.hexa3} and the mesh data is shown in Table \ref{table:mesh.data}.
\begin{figure}[h]
	\centering
			\includegraphics[width=4\textwidth/9]{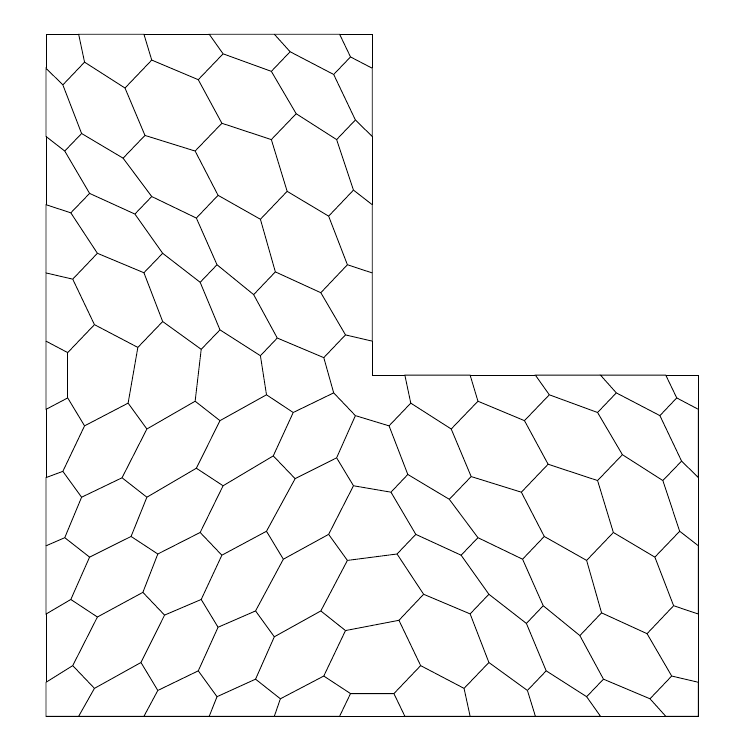} 
			\includegraphics[width=4\textwidth/9]{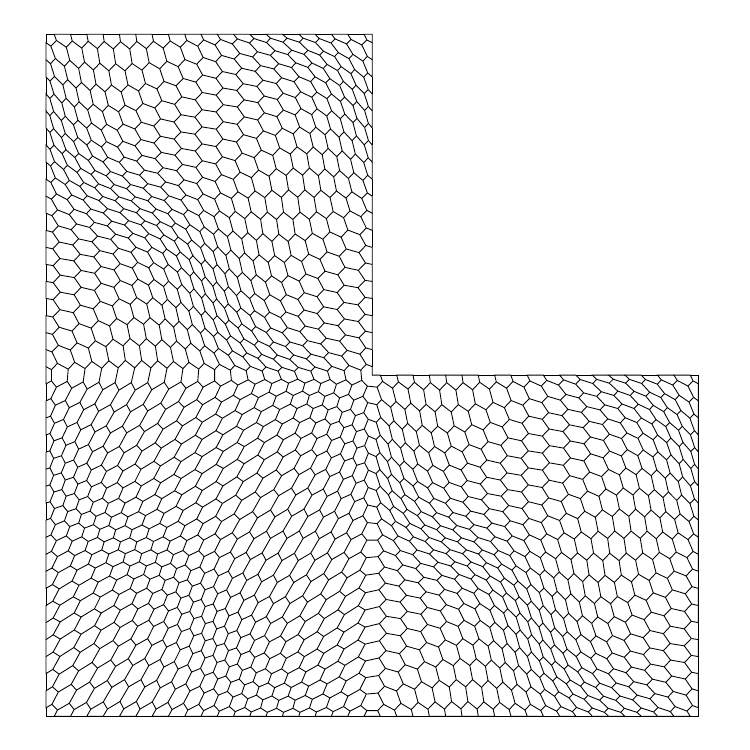} 
	\caption{Two members of the mesh sequence}
	\label{fig:mesh.Lshape.hexa3}
\end{figure}
\begin{table}[h]
	\centering
	\pgfplotstableread{data/unenriched_k2.dat}\loadedtable
	\pgfplotstabletypeset
	[
	columns={MeshSize,NbCells,NbInternalEdges}, 
	columns/MeshSize/.style={column name=$h$,/pgf/number format/.cd,fixed zerofill,precision=4},
	columns/NbCells/.style={column name=Nb. Elements},
	columns/NbInternalEdges/.style={column name=Nb. Internal Edges},
	every head row/.style={before row=\toprule,after row=\midrule},
	every last row/.style={after row=\bottomrule} 
	]\loadedtable
	\caption{Parameters of the mesh sequence}
	\label{table:mesh.data}
\end{table}

The conditioning of the system poses a significant challenge to the scheme. This is a common problem for enriched schemes and the matter is discussed in detail for the enriched NCVEM \cite{artioli.mascotto:2021:enrichment}. For a \(T\in\Th\), let us denote by \( \{\phi_T^j\}_{j=1}^{\dim\{\POLYX{k+1}(T)\}} \) the basis for \(\POLYX{k+1}(T)\) prior to orthonormalising. The mass matrix is defined by
\begin{equation}
	(\mat{M}_T)_{i,j} = \brac[T]{\phi_T^i,\phi_T^j}. \nn
\end{equation}
The maximum condition number of the element mass matrices is defined by \(\mathcal{C}\defeq\max_{T\in\Th} \frac{\lambda_{T,max}}{\lambda_{T,min}}\), where \(\lambda_{T,max}\) and \(\lambda_{T,min}\) denote the maximum and minimum eigen values of \(\mat{M}_T\) respectively. The parameter \(\mathcal{C}\) gives a measure of how linearly independent the worst performing basis is. In Table \ref{table:condition.number} we present the maximum condition number on the finest mesh in the sequence with various values of \(k\) and \(\gamma\) (where \(\gamma=0\) corresponds to an non-enriched scheme). It is clear that the conditioning of the scheme is significantly worse for enriched schemes than non-enriched schemes and gets progressively worse with increasing \(k\) and \(\gamma\). We note that once \(\mathcal{C}\sim 10^{16}\) the orthonormalisation process (and thus the scheme itself) fails due to division by numbers which are numerically zero.

\begin{table}[H]
	\centering
	
	\pgfplotstableread{data/condition-0.dat}\loadedtable	
	\pgfplotstabletypeset
	[
	columns={Gamma,CellCondition},
	columns/CellCondition/.style={column type/.add={|}{},column name=\(k\eqs0\),/pgf/number format/sci,zerofill,precision=3},
	columns/Gamma/.style={string type, column type = {l},column name=},
	every head row/.style={after row=\midrule}
	]\loadedtable	
	\pgfplotstableread{data/condition-1.dat}\loadedtable
	\pgfplotstabletypeset
	[
	columns={CellCondition},
	columns/CellCondition/.style={column name=\(k\eqs1\),/pgf/number format/sci,zerofill,precision=3},
	every head row/.style={after row=\midrule}
	]\loadedtable	
	\pgfplotstableread{data/condition-2.dat}\loadedtable
	\pgfplotstabletypeset
	[
	columns={CellCondition},
	columns/CellCondition/.style={column name=\(k\eqs2\),/pgf/number format/sci,zerofill,precision=3},
	every head row/.style={after row=\midrule}
	]\loadedtable	
	\pgfplotstableread{data/condition-3.dat}\loadedtable
	\pgfplotstabletypeset
	[
	columns={CellCondition},
	columns/CellCondition/.style={column name=\(k\eqs3\),/pgf/number format/sci,zerofill,precision=3},
	every head row/.style={after row=\midrule}
	]\loadedtable
	
	\caption{Maximum local condition number with \(h\approx 0.0257\)}
	\label{table:condition.number}
\end{table}

Denote by \(\uluXh \in\UXhkzr\) the exact solution to the discrete problem \eqref{eq:discrete.problem}. The relative error of the scheme is determined via the following three quantities,
\begin{align}
E_{0,\Th} \deq \Big[\frac{\sum_{T\in\Th}\norm[T]{\uXT-\piXTzr{k} u}^2}{\sum_{T\in\Th}\norm[T]{\piXTzr{k} u}^2}\Big]^\frac12 + \Big[\frac{\sum_{F\in\Fh}h_F\norm[F]{\uXF-\piXFzr{k} u}^2}{\sum_{F\in\Fh}h_F\norm[F]{\piXFzr{k} u}^2}\Big]^\frac12 \nl
E_{1,\Th} \deq \Big[\frac{\sum_{T\in\Th}\seminorm[H^1(T)]{\pXT{k+1}\uluXT - \piXTe{k+1}  u}^2}{\sum_{T\in\Th}\seminorm[H^1(T)]{\piXTe{k+1}  u}^2}\Big]^\frac12 \nl
E_{\a,\xh} \deq \frac{\energynorm{\uluXh - \IXhk u}}{\energynorm{\IXhk u}}. \nn
\end{align}

Unlike \cite{artioli.mascotto:2021:enrichment}, we do not observe saturation in error rates as the condition number gets large, but rather an instantaneous failure of the scheme. As such we would like to choose a cut-off value \(\gamma\) that is as large as possible and does not result in system failure. In Figures \ref{fig:test.k1} and \ref{fig:test.k2} we test convergence of the scheme with \(k=1\) and \(k=2\) respectively. The enriched scheme performs significantly better than the non-enriched scheme, particularly in \(H^1\)-error (\(E_{1,\Th}\)) and for higher polynomial degree \(k\). In Figure \ref{fig:test.k2}, we test two cut-off values (\(\gamma=0.075\) and \(\gamma=0.15\)). The larger cut-off value performs slightly better in \(H^1\)-error, however, the scheme fails on the final mesh.

\begin{figure}[H]
	\centering
	\ref{k1}
	\vspace{0.5cm}\\
	\subcaptionbox{$E_{0,\Th}$ vs $h$}
	{
		\begin{tikzpicture}[scale=0.56]
		\begin{loglogaxis}[ legend columns=3, legend to name=k1 ]
		\addplot table[x=MeshSize,y=L2Error] {data/unenriched_k1.dat};
		\addplot table[x=MeshSize,y=L2Error] {data/local_k1_cut_015.dat};
		\logLogSlopeTriangle{0.90}{0.4}{0.1}{3}{black};
		\legend{Non-enriched, Locally enriched $(\gamma = 0.15)$};
		\end{loglogaxis}
		\end{tikzpicture}
	}
	\subcaptionbox{$E_{1,\Th}$ vs $h$}
	{
		\begin{tikzpicture}[scale=0.56]
		\begin{loglogaxis}
		\addplot table[x=MeshSize,y=H1Error] {data/unenriched_k1.dat};
		\addplot table[x=MeshSize,y=H1Error] {data/local_k1_cut_015.dat};
		\logLogSlopeTriangle{0.90}{0.4}{0.1}{2}{black};
		\end{loglogaxis}
		\end{tikzpicture}
	}
	\subcaptionbox{$E_{\a,\xh}$ vs $h$}
	{
		\begin{tikzpicture}[scale=0.56]
		\begin{loglogaxis}
		\addplot table[x=MeshSize,y=EnergyError] {data/unenriched_k1.dat};
		\addplot table[x=MeshSize,y=EnergyError] {data/local_k1_cut_015.dat};
		\logLogSlopeTriangle{0.90}{0.4}{0.1}{2}{black};
		\end{loglogaxis}
		\end{tikzpicture}
	}
	\caption{Error vs $h$, \(k=1\)}
	\label{fig:test.k1}
\end{figure}
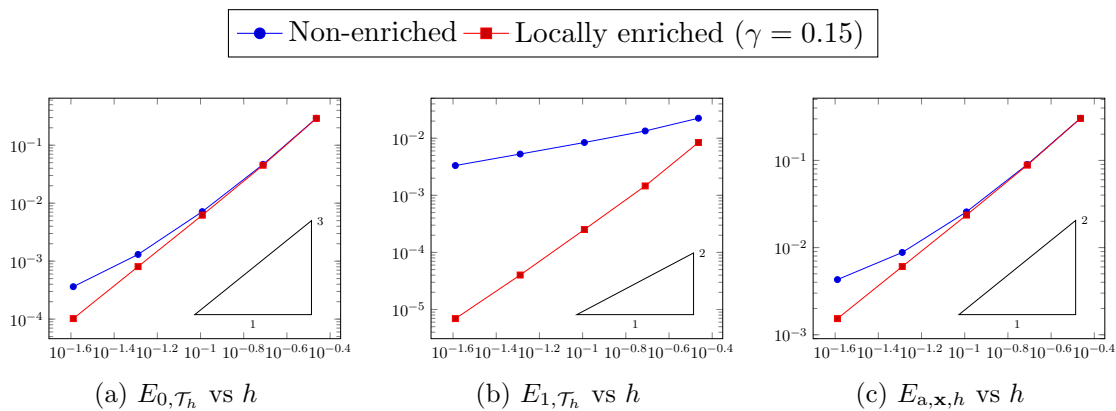

\begin{figure}[H]
	\centering
	\ref{k2}
	\vspace{0.5cm}\\
	\subcaptionbox{$E_{0,\Th}$ vs $h$}
	{
		\begin{tikzpicture}[scale=0.56]
		\begin{loglogaxis}[ legend columns=3, legend to name=k2 ]
		\addplot table[x=MeshSize,y=L2Error] {data/unenriched_k2.dat};
		\addplot table[x=MeshSize,y=L2Error] {data/local_k2_cut_0075.dat};
		\addplot table[x=MeshSize,y=L2Error] {data/local_k2_cut_015.dat};
		\logLogSlopeTriangle{0.90}{0.4}{0.1}{4}{black};
		\legend{Non-enriched, Locally enriched ($\gamma = 0.075$), Locally Enriched ($\gamma = 0.15$)};
		\end{loglogaxis}
		\end{tikzpicture}
	}
	\subcaptionbox{$E_{1,\Th}$ vs $h$}
	{
		\begin{tikzpicture}[scale=0.56]
		\begin{loglogaxis}
		\addplot table[x=MeshSize,y=H1Error] {data/unenriched_k2.dat};
		\addplot table[x=MeshSize,y=H1Error] {data/local_k2_cut_0075.dat};
		\addplot table[x=MeshSize,y=H1Error] {data/local_k2_cut_015.dat};
		\logLogSlopeTriangle{0.90}{0.4}{0.1}{3}{black};
		\end{loglogaxis}
		\end{tikzpicture}
	}
	\subcaptionbox{$E_{\a,\xh}$ vs $h$}
	{
		\begin{tikzpicture}[scale=0.56]
		\begin{loglogaxis}
		\addplot table[x=MeshSize,y=EnergyError] {data/unenriched_k2.dat};
		\addplot table[x=MeshSize,y=EnergyError] {data/local_k2_cut_0075.dat};
		\addplot table[x=MeshSize,y=EnergyError] {data/local_k2_cut_015.dat};
		\logLogSlopeTriangle{0.90}{0.4}{0.1}{3}{black};
		\end{loglogaxis}
		\end{tikzpicture}
	}
	\caption{Error vs $h$, \(k=2\)}
	\label{fig:test.k2}
\end{figure}
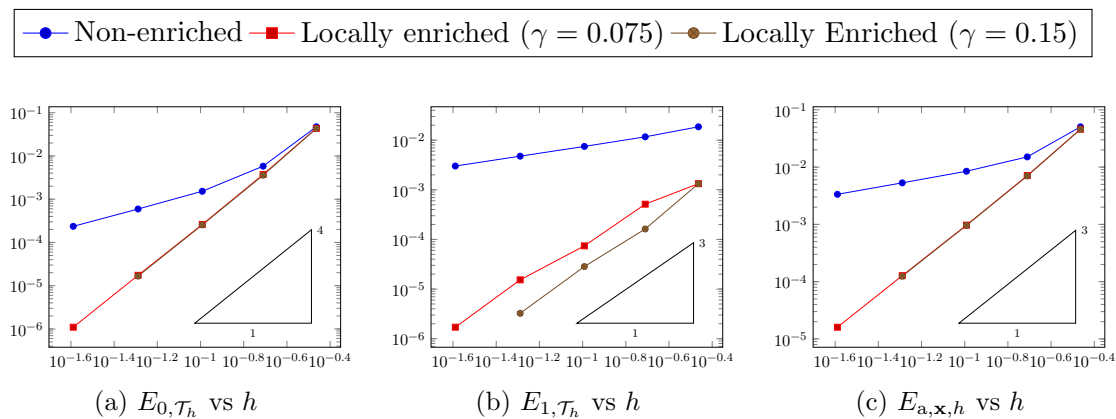

A standard error estimate with respect to $k$ for a classical HHO scheme is given by
\begin{equation}\label{eq:error.estimate.k}
\norm[\a,h]{\uluh - \Ihk u} \lesssim \frac{h^{k+1}}{(k+1)^k} \seminorm[\HS{k+2}(\Th)]{u}
\end{equation}
where the hidden constant is independent of $k$ \cite[c.f.][Theorem 3.3]{aghili.di-pietro.ea:2017:hp}. Therefore, if
\begin{equation}\label{eq:sovolev.seminorm.bound}
	\frac{\seminorm[\HS{k+2}(\Th)]{u}}{(k+1)^k} \le C
\end{equation}
for some $C$ independent of $k$, then the scheme should converge exponentially with respect to $k$. 
In Figure \ref{fig:test.h2}, we fix the mesh (the second mesh in Table \ref{table:mesh.data}) and test convergence as \(k\) increases. In this case, the polynomial spaces are enriched only on the element containing the singular point and on its faces. While convergence with respect to $k$ has not been proven in this work, equation \eqref{eq:error.estimate.k} provides a bench mark to compare the tests to. As such, we plot on a log-linear scale and include a line of slope $\log h$. Again, it is quite clear that the enriched scheme (albeit enriched on a single element) performs much better than the non-enriched scheme. However, it is apparent that even the enriched scheme does not converge exponentially. By only enriching the element containing the singular point, there exist non-enriched elements which are `close' to the singularity. As such, the Sobolev seminorms will grow quite quickly as $k$ increases. Therefore, a large $k$ may be required before the boundedness \eqref{eq:sovolev.seminorm.bound} is apparent. This explains why exponential convergence of the enriched scheme is not observed in Figure \ref{fig:test.h2}.

\begin{figure}[H]
	\centering
	\ref{ktest}
	\vspace{0.5cm}\\
	\subcaptionbox{$E_{0,\Th}$ vs $k$}
	{
		\begin{tikzpicture}[scale=0.57]
		\begin{semilogyaxis}[ legend columns=3, legend to name=ktest ]
		\addplot table[x=EdgeDegree,y=L2Error] {data/ktest_hexa2_unenriched.dat};
		\addplot table[x=EdgeDegree,y=L2Error] {data/ktest_hexa2_enriched.dat};
		\addplot[mark=none,densely dotted,blue] table[x=EdgeDegree,y=LogH_slope] {data/test_file.dat};
		\legend{Non-enriched, Enriched on a single element, $h^k$};
		\end{semilogyaxis}
		\end{tikzpicture}
	}
	\subcaptionbox{$E_{1,\Th}$ vs $k$}
	{
		\begin{tikzpicture}[scale=0.57]
		\begin{semilogyaxis}
		\addplot table[x=EdgeDegree,y=H1Error] {data/ktest_hexa2_unenriched.dat};
		\addplot table[x=EdgeDegree,y=H1Error] {data/ktest_hexa2_enriched.dat};
		\addplot[mark=none,densely dotted,blue] table[x=EdgeDegree,y=LogH_slope] {data/test_file.dat};
		\end{semilogyaxis}
		\end{tikzpicture}
	}
	\subcaptionbox{$E_{\a,\xh}$ vs $k$}
	{
		\begin{tikzpicture}[scale=0.57]
		\begin{semilogyaxis}
		\addplot table[x=EdgeDegree,y=EnergyError] {data/ktest_hexa2_unenriched.dat};
		\addplot table[x=EdgeDegree,y=EnergyError] {data/ktest_hexa2_enriched.dat};
		\addplot[mark=none,densely dotted,blue] table[x=EdgeDegree,y=LogH_slope] {data/test_file.dat};
		\end{semilogyaxis}
		\end{tikzpicture}
	}
	\caption{Error vs $k$, $h\approx 0.195$}
	\label{fig:test.h2}
\end{figure}
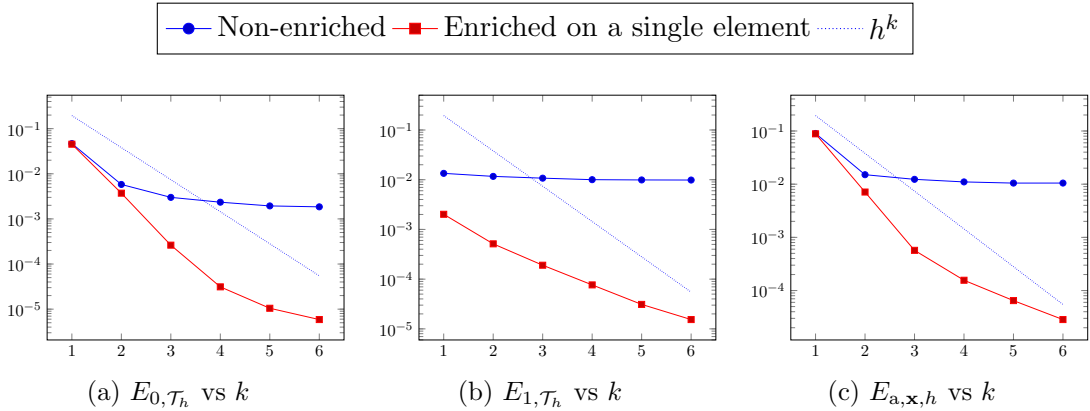

In Figure \ref{fig:cartesian.mesh} we consider two local enrichment schemes on a uniform Cartesian mesh with 48 elements. The elements which are enriched with the singular functions are bordered in red. Due to the larger element diameters and `rounder' element geometries we are able to consider a much larger radius of enrichment without ill-conditioning causing the scheme to fail for large $k$. As such, the non-enriched elements are further from the singular point (particularly in scheme 2) so we expect the high-order Sobolev seminorms to be smaller. We plot convergence of the schemes with respect to $k$ in Figure \ref{fig:ktest.cartesian}. It appears that the locally enriched scheme 2 maintains exponential convergence with respect to $k$ in both $L^2$ and energy error. This is similar to results in enriched NCVEM \cite[Figure 18]{artioli.mascotto:2021:enrichment} where exponential convergence of a globally enriched scheme on a coarse Voronoi mesh is observed. While the $H^1$ error of scheme 2 does observe a minor saturation in convergence rate, the rate is consistently better than $h^k$.  It appears that the convergence rate of scheme 1 becomes slightly sub-optimal as $k$ gets large due to the existence of non-enriched elements closer to the singular point. 

\begin{figure}[H]
	\centering
	\subcaptionbox{Local enrichment scheme 1.}
	{
		\includegraphics[width=0.3\textwidth]{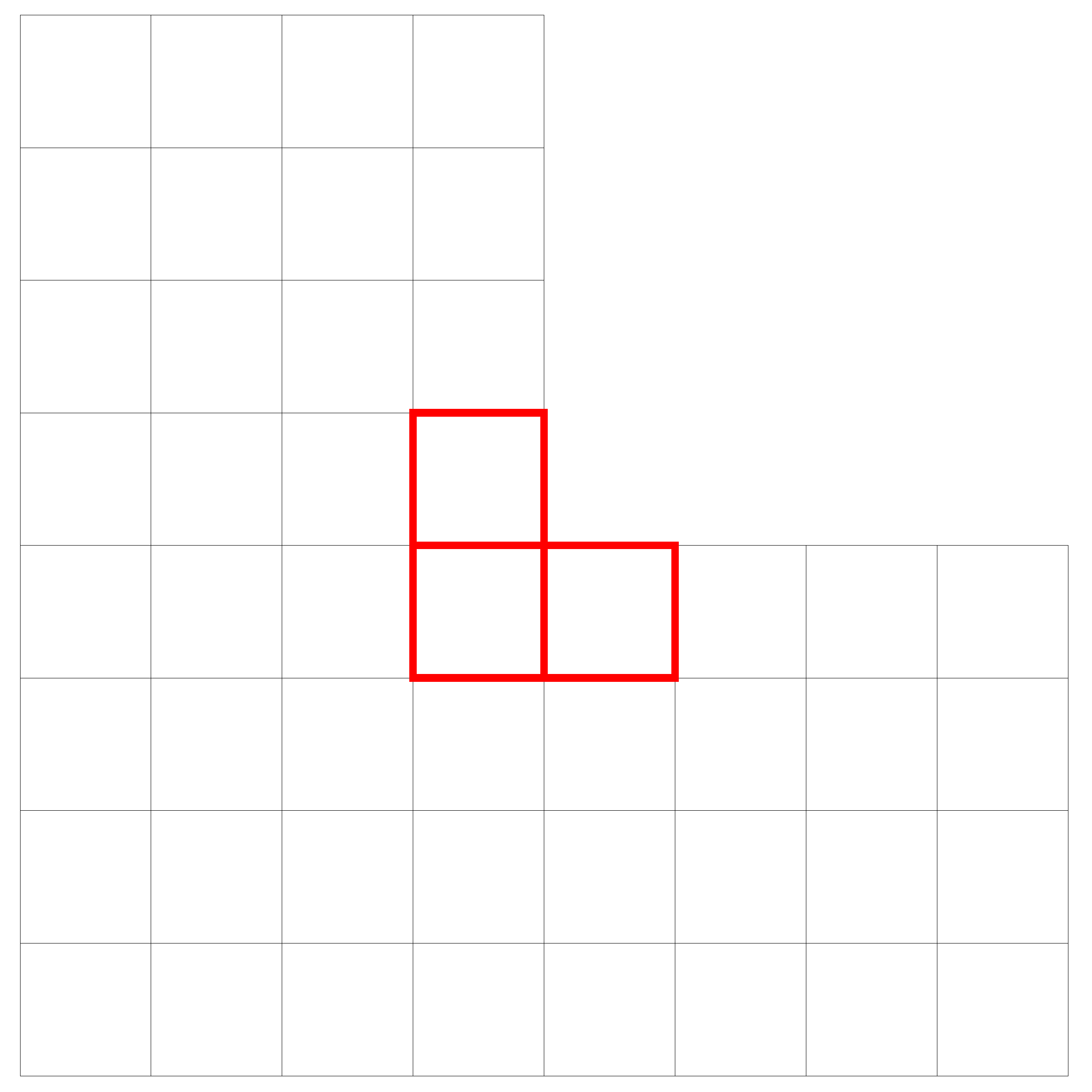}
	}
${}\qquad{}$
	\subcaptionbox{Local enrichment scheme 2.}
	{
	\includegraphics[width=0.3\textwidth]{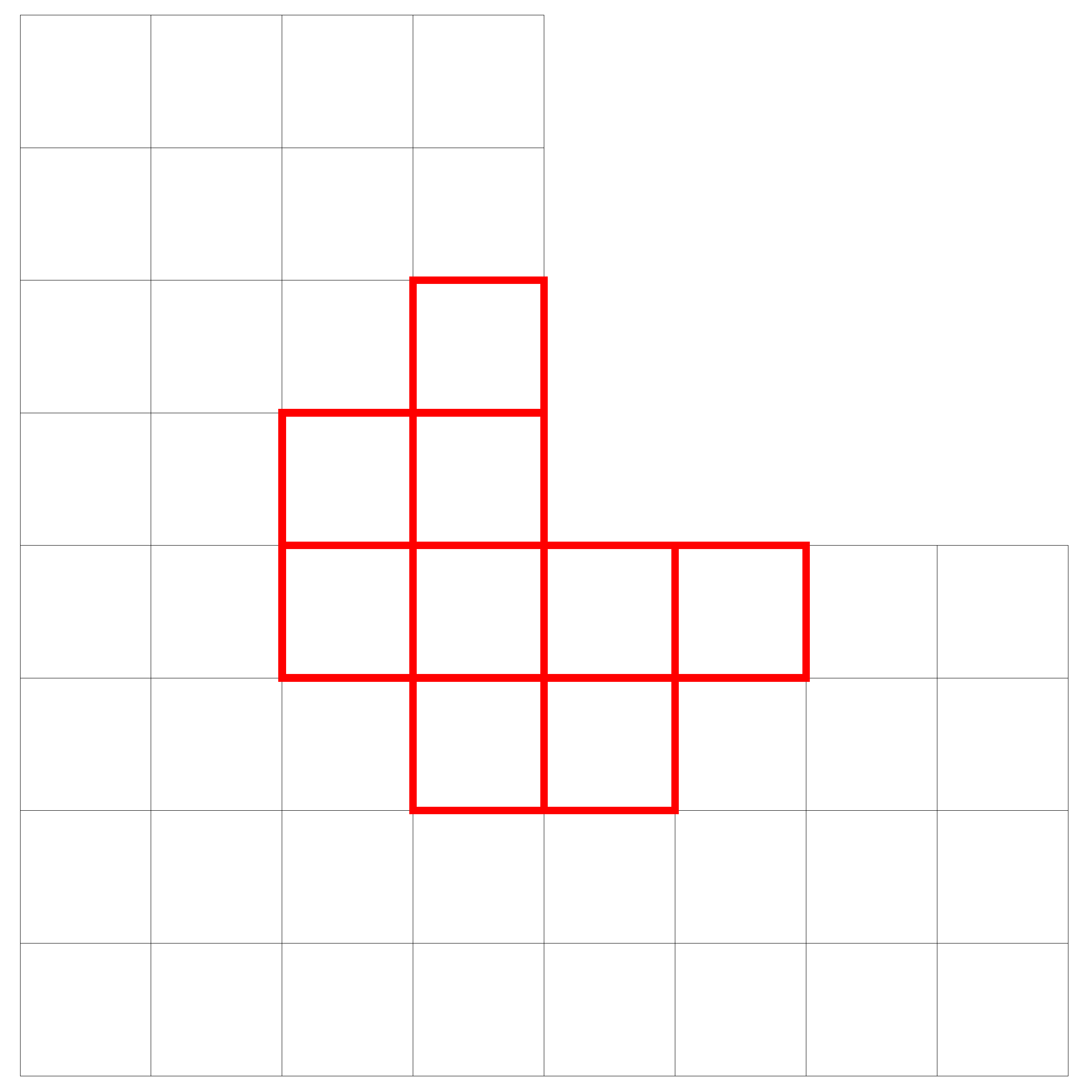}
	}
\caption{Uniform Cartesian mesh with enriched elements bordered in red}
\label{fig:cartesian.mesh}
\end{figure}

\begin{figure}[H]
\centering
\ref{ktest.cartesian}
\vspace{0.5cm}\\
\subcaptionbox{$E_{0,\Th}$ vs $k$}
{
	\begin{tikzpicture}[scale=0.57]
	\begin{semilogyaxis}[ legend columns=4, legend to name=ktest.cartesian, xtick={0,1,2,3,4,5,6,7}, xticklabels={0,1,2,3,4,5,6,7} ]
	\addplot table[x=EdgeDegree,y=L2Error] {data/ktest_cartesian_2_nonenriched.dat};
	\addplot table[x=EdgeDegree,y=L2Error] {data/ktest_cartesian_2_locally_enriched_1.dat};
	\addplot table[x=EdgeDegree,y=L2Error] {data/ktest_cartesian_2_locally_enriched_2.dat};
	\addplot[mark=none,densely dotted,blue] table[x=EdgeDegree,y=LogH_slope] {data/logh_cartesian_2.dat};
	\legend{Non-enriched, Locally enriched (scheme 1), Locally enriched (scheme 2), $h^k$};
	\end{semilogyaxis}
	\end{tikzpicture}
}
\subcaptionbox{$E_{1,\Th}$ vs $k$}
{
	\begin{tikzpicture}[scale=0.57]
	\begin{semilogyaxis}[xtick={0,1,2,3,4,5,6,7}, xticklabels={0,1,2,3,4,5,6,7}]
	\addplot table[x=EdgeDegree,y=H1Error] {data/ktest_cartesian_2_nonenriched.dat};
	\addplot table[x=EdgeDegree,y=H1Error] {data/ktest_cartesian_2_locally_enriched_1.dat};
	\addplot table[x=EdgeDegree,y=H1Error] {data/ktest_cartesian_2_locally_enriched_2.dat};
	\addplot[mark=none,densely dotted,blue] table[x=EdgeDegree,y=LogH_slope] {data/logh_cartesian_2.dat};
	\end{semilogyaxis}
	\end{tikzpicture}
}
\subcaptionbox{$E_{\a,\xh}$ vs $k$}
{
	\begin{tikzpicture}[scale=0.57]
	\begin{semilogyaxis}[xtick={0,1,2,3,4,5,6,7}, xticklabels={0,1,2,3,4,5,6,7}]
	\addplot table[x=EdgeDegree,y=EnergyError] {data/ktest_cartesian_2_nonenriched.dat};
	\addplot table[x=EdgeDegree,y=EnergyError] {data/ktest_cartesian_2_locally_enriched_1.dat};
	\addplot table[x=EdgeDegree,y=EnergyError] {data/ktest_cartesian_2_locally_enriched_2.dat};
	\addplot[mark=none,densely dotted,blue] table[x=EdgeDegree,y=LogH_slope] {data/logh_cartesian_2.dat};
	\end{semilogyaxis}
	\end{tikzpicture}
}
\caption{Error vs $k$, Cartesian mesh, $h\approx 0.354$}
\label{fig:ktest.cartesian}
\end{figure}
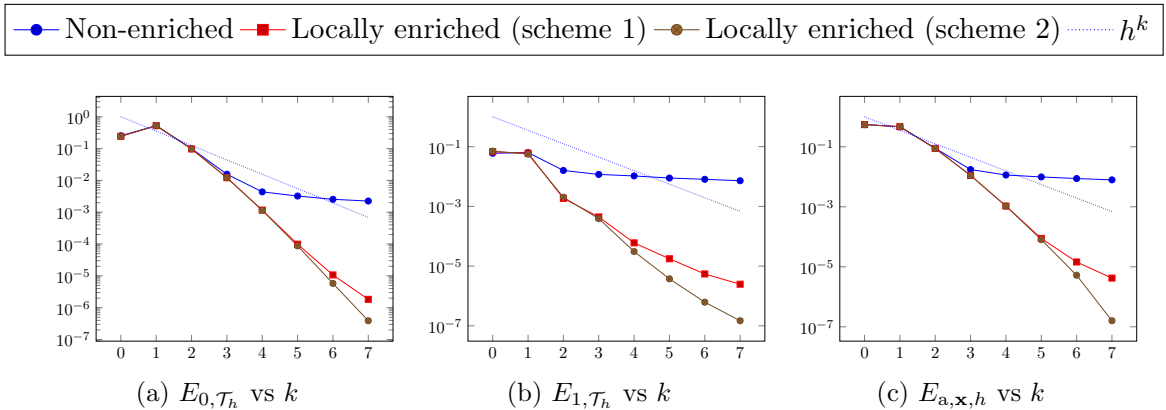

\subsection{Oscillatory Solution}\label{sec:oscillatory}
We briefly consider here an application of the XHHO scheme to highly oscillatory solutions of problem \eqref{eq:weak.form}. If the source term consists of a highly oscillatory component, it is expected that the solution will reflect this. In standard HHO, such details can only be captured with either high-order approximations or fine meshes. The XHHO method, however, is able to obtain improved error estimates by enriching the local spaces with an oscillatory function.

Let us consider here a square domain $\Omega=(0,1)^2$ and exact solution
\begin{equation}\label{eq:oscillatory.solution}
	u = \sin(\pi x_1) \sin(\pi x_2) + \sin(\frac{1}{\hat{r}^2 + \epsilon})
\end{equation}
with $\hat{r}^2=(x_1 - 0.5)^2 + (x_2 - 0.5)^2$ denoting the radial distance to the domain center and $\epsilon > 0$ a small constant. We note that equation \eqref{eq:oscillatory.solution} does not actually define a singular solution, thus the standard HHO method is expected to converge optimally. However, as the multiplicative constants in the error depend on the $(k+2)$-th derivative of the solution, the error is expected to get very large as $\epsilon \to 0$.

The oscillatory component of \eqref{eq:oscillatory.solution} is denoted by
\[
	\psi \defeq \sin(\frac{1}{\hat{r}^2 + \epsilon}).
\]
The enrichment space is naturally defined as $W(T) = \textrm{span}\{\psi|_T\}$. We note that as $\Delta \psi \ne 0$ we are no longer able to consider $\POLYD{k}(T) = \POLY{k}(T)$. We take $\epsilon = 0.05$ and consider both non-enriched and locally enriched schemes (here, the local enrichment process is identical to that described in Section \ref{sec:local.enrichment}, except with the singular point taken as $\hat{r}=0$ rather than $r=0$). We consider a sequence of regular, triangular meshes and plot the results in Figures \ref{fig:oscillator_k0} and \ref{fig:oscillator_k1} for $k=0,1$. It is clear that while both schemes converge optimally, the absolute error of the non-enriched scheme can be several orders of magnitude worse than that of the locally enriched scheme.
\begin{figure}[H]
	\centering
	\ref{oscillator_k0}
	\vspace{0.5cm}\\
	\subcaptionbox{$E_{0,\Th}$ vs $h$}
	{
		\begin{tikzpicture}[scale=0.57]
		\begin{loglogaxis}[ legend columns=2, legend to name=oscillator_k0 ]
		\addplot table[x=MeshSize,y=L2Error] {data/oscillator_tri_mesh_unenriched_k0.dat};
		\addplot table[x=MeshSize,y=L2Error] {data/oscillator_tri_mesh_cut_05_k0.dat};
		\logLogSlopeTriangle{0.90}{0.4}{0.1}{2}{black};
		\legend{Non-enriched, Locally enriched ($\gamma=0.5$)};
		\end{loglogaxis}
		\end{tikzpicture}
	}
	\subcaptionbox{$E_{1,\Th}$ vs $h$}
	{
		\begin{tikzpicture}[scale=0.57]
		\begin{loglogaxis}
		\addplot table[x=MeshSize,y=H1Error] {data/oscillator_tri_mesh_unenriched_k0.dat};
		\addplot table[x=MeshSize,y=H1Error] {data/oscillator_tri_mesh_cut_05_k0.dat};
		\logLogSlopeTriangle{0.90}{0.4}{0.1}{1}{black};
		\end{loglogaxis}
		\end{tikzpicture}
	}
	\subcaptionbox{$E_{\a,\xh}$ vs $h$}
	{
		\begin{tikzpicture}[scale=0.57]
		\begin{loglogaxis}
		\addplot table[x=MeshSize,y=EnergyError] {data/oscillator_tri_mesh_unenriched_k0.dat};
		\addplot table[x=MeshSize,y=EnergyError] {data/oscillator_tri_mesh_cut_05_k0.dat};
		\logLogSlopeTriangle{0.90}{0.4}{0.1}{1}{black};
		\end{loglogaxis}
		\end{tikzpicture}
	}
	\caption{Error vs $h$, $k=0$}
	\label{fig:oscillator_k0}
\end{figure}
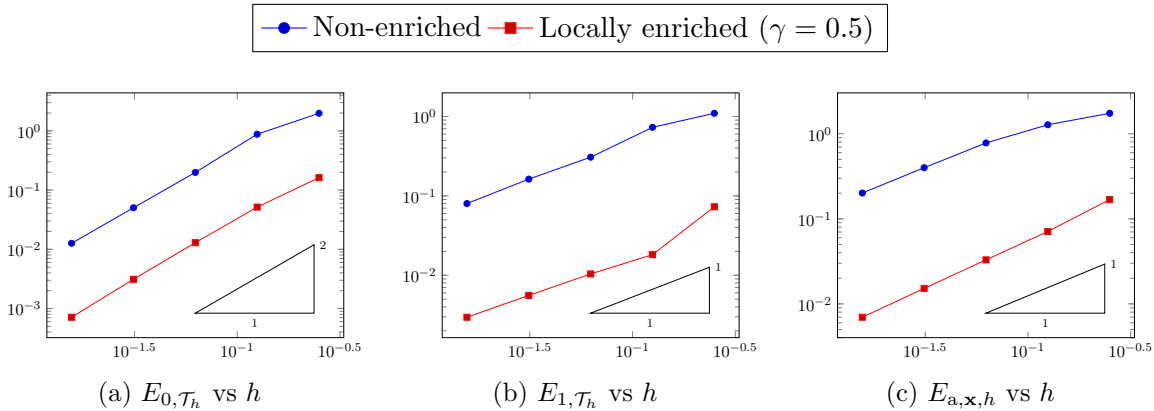

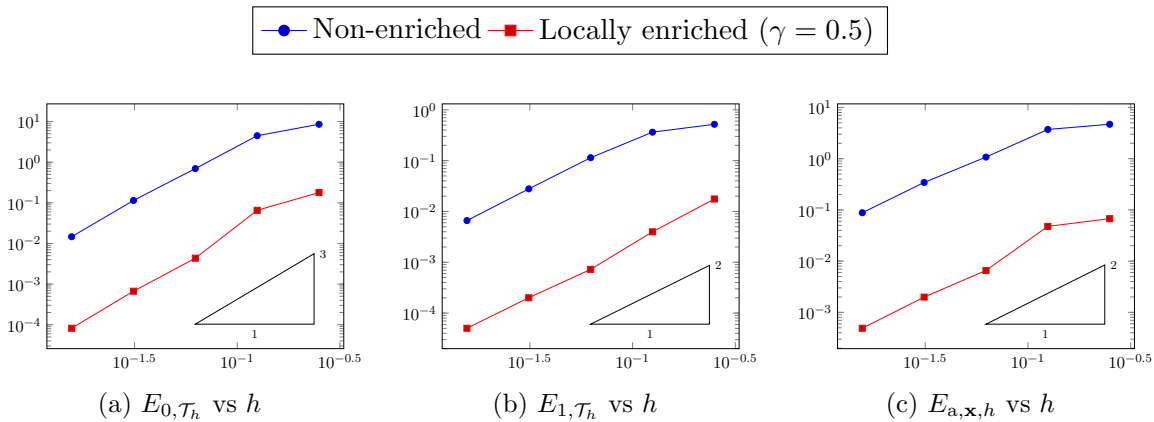
\begin{figure}[H]
\centering
\ref{oscillator_k1}
\vspace{0.5cm}\\
\subcaptionbox{$E_{0,\Th}$ vs $h$}
{
	\begin{tikzpicture}[scale=0.57]
	\begin{loglogaxis}[ legend columns=2, legend to name=oscillator_k1 ]
	\addplot table[x=MeshSize,y=L2Error] {data/oscillator_tri_mesh_unenriched_k1.dat};
	\addplot table[x=MeshSize,y=L2Error] {data/oscillator_tri_mesh_cut_05_k1.dat};
	\logLogSlopeTriangle{0.90}{0.4}{0.1}{3}{black};
	\legend{Non-enriched, Locally enriched ($\gamma=0.5$)};
	\end{loglogaxis}
	\end{tikzpicture}
}
\subcaptionbox{$E_{1,\Th}$ vs $h$}
{
	\begin{tikzpicture}[scale=0.57]
	\begin{loglogaxis}
	\addplot table[x=MeshSize,y=H1Error] {data/oscillator_tri_mesh_unenriched_k1.dat};
	\addplot table[x=MeshSize,y=H1Error] {data/oscillator_tri_mesh_cut_05_k1.dat};
	\logLogSlopeTriangle{0.90}{0.4}{0.1}{2}{black};
	\end{loglogaxis}
	\end{tikzpicture}
}
\subcaptionbox{$E_{\a,\xh}$ vs $h$}
{
	\begin{tikzpicture}[scale=0.57]
	\begin{loglogaxis}
	\addplot table[x=MeshSize,y=EnergyError] {data/oscillator_tri_mesh_unenriched_k1.dat};
	\addplot table[x=MeshSize,y=EnergyError] {data/oscillator_tri_mesh_cut_05_k1.dat};
	\logLogSlopeTriangle{0.90}{0.4}{0.1}{2}{black};
	\end{loglogaxis}
	\end{tikzpicture}
}
\caption{Error vs $h$, $k=1$}
\label{fig:oscillator_k1}
\end{figure}

\section{Conclusion}

In this paper we have introduced and analysed the Extended Hybrid High-Order method for the Poisson problem. The method is shown to be robust, and capable of handling generic singularities satisfying Assumption \ref{assum:regularity} in arbitrary dimensions. Optimal error estimates are established in both $H^1$- and discrete energy norms. The error analysis is backed up by numerical simulations which show the Extended Hybrid High-Order method to be a viable technique for handling irregular solutions. The method is also applicable to any situation where some aspects of the solution can be (at least locally) determined. This is seen in Section \ref{sec:oscillatory} where the XHHO method is applied to solutions possessing a highly oscillatory component. Indeed, the enrichment function considered is actually smooth, yet a considerable improvement in performance was observed which highlights the versatility of the method for any cases where there is prior knowledge of the solution behaviour. The scheme also has natural generalisations to more general linear elliptic problems of the form $Lu=f$.

The case of corner singularities is discussed and tested in detail in Section \ref{sec:numerical}. The schemes are shown to converge optimally and perform significantly better than standard HHO for solutions possessing a weak singularity at a re-entrant corner in two dimensions. The XHHO method is capable of handling such irregularities for all but the limiting case of the slit domain. For such cases, further work into developing a robust Hybrid High-Order method is required. In three dimensions, the singularities arising from irregular geometries can be more complicated. Along an edge which forms a non-convex corner in the domain, the singularity can be written as the product of a smooth function and a weakly singular function whose behaviour perpendicular to the edge is described by the two-dimensional case (c.f. \cite{grisvard:1992:singularities}). Again, this scheme is capable of modelling such irregularity for all but the slit domain. 



\section*{Declarations}
The author declares that they have no conflict of interest.



\bibliographystyle{plain}
\bibliography{references}

\end{document}